\documentclass[
12pt]{amsart}
\usepackage{amscd}
\usepackage{amssymb}
\usepackage{a4wide}
\usepackage{amstext}
\usepackage{mathrsfs}
\usepackage[final]{hyperref}
\hypersetup{unicode= false, colorlinks=true, linkcolor=blue,
anchorcolor=blue, citecolor=green, filecolor=red, menucolor=blue,
pagecolor=blue, urlcolor=blue} 
\numberwithin{equation}{section}

\newcommand{\half}{\frac{1}{2}}

\newcommand{\eps}{\varepsilon}

\newcommand{\CC}{\mathbb {C}}
\newcommand{\Hg}{\mathscr{H}}

\newtheorem{theorem}{Theorem}[section]
\newtheorem{thm}{Theorem}[section]

\newtheorem{lemma}{Lemma}[section]
\newtheorem{corollary}{Corollary}[section]

\newtheorem*{claim}{Claim}
\newtheorem*{question}{Question}

\newtheorem*{example}{Example 1}
\newtheorem*{example2}{Example 2}

\title[Discrete Hilbert transforms on sparse sequences]
{Discrete Hilbert transforms on sparse sequences}
\author [Yurii Belov]{Yurii Belov }
\address{Department of Mathematical Sciences\\
Norwegian University of Science and Technology (NTNU)\\
 NO- 7491 Trondheim, Norway}
 \email{j\_b\_juri\_belov@mail.ru}
\author [Tesfa Y. Mengestie]{Tesfa Y. Mengestie }
\address{Department of Mathematical Sciences\\
Norwegian University of Science and Technology (NTNU)\\
 NO- 7491 Trondheim, Norway}
\email{mengesti@math.ntnu.no}
\author[Kristian Seip]{Kristian Seip}
\address{Department of Mathematical Sciences\\
Norwegian University of Science and Technology (NTNU)\\
 NO- 7491 Trondheim, Norway}
\email{seip@math.ntnu.no}
\thanks{The authors are supported by the Research Council of
Norway grant 185359/V30.} 
\begin{document}
\begin{abstract}
Weighted discrete Hilbert transforms $(a_n)_n \mapsto \sum_n a_n
v_n/(z-\gamma_n)$ from $\ell^2_v$ to a weighted $L^2$ space are
studied, with $\Gamma=(\gamma_n)$ a sequence of distinct points in
the complex plane and $v=(v_n)$ a corresponding sequence of positive
numbers. In the special case when $|\gamma_n|$ grows at least
exponentially, bounded transforms of this kind are described in
terms of a simple relative to the Muckenhoupt $(A_2)$ condition. The
special case when $z$ is restricted to another sequence $\Lambda$ is
studied in detail; it is shown that a bounded transform satisfying a
certain admissibility condition can be split into finitely many
surjective transforms, and precise geometric conditions are found
for invertibility of such two weight transforms. These results can
be interpreted as statements about systems of reproducing kernels in
certain Hilbert spaces of which de Branges spaces and model
subspaces of $H^2$ are prime examples. In particular, a connection
to the Feichtinger conjecture is pointed out. Descriptions of
Carleson measures and Riesz bases of normalized reproducing kernels
for certain ``small'' de Branges spaces follow from the results of
this paper.
\end{abstract}
\maketitle

\section{Introduction and main results}

This paper is concerned with the mapping properties of what we will
call discrete Hilbert transforms in the complex plane. One aspect of
this topic was treated in \cite{BMS}, where all unitary discrete
Hilbert transforms were described. When we now turn to questions
about boundedness, surjectivity, and invertibility, results of the
same generality seem at present out of reach. The results to be
presented below are complete only when the discrete Hilbert
transforms are defined on particularly sparse sequences. We will
nevertheless present the problems in the most general setting, as we
believe they merit further investigations. We will emphasize the
connection with topics such as Carleson measures and Riesz bases of
normalized reproducing kernels in Hilbert spaces of analytic
functions; this will lead us to the most intriguing general
question, namely whether or not the Feichtinger conjecture holds
true for systems of reproducing kernels in such spaces.

We begin by assuming that we are given an infinite sequence of
distinct points $\Gamma=(\gamma_n)$ in $\CC$ and a corresponding
sequence of positive numbers $v=(v_n)$, both indexed by the positive
integers. We define the weighted Hilbert transform as the map
\begin{equation} (a_n)\mapsto\sum_{n=1}^\infty \frac{a_n
v_n}{z-\gamma_n}, \label{discreteH}
\end{equation}
which is well defined when $(a_n)$ belongs to $\ell^2_v=\{ (a_n):\
\|a\|_v^2= \sum_{n=1}^\infty |a_n|^2 v_n<\infty\}$ and $z$ is a
point in the set
\[ (\Gamma, v)^{*}=\left\{ z\in \CC:\
\sum_{n=1}^\infty \frac{v_n}{|z-\gamma_n|^2}<\infty\right\}. \] We
denote the transformation defined in \eqref{discreteH} by
$H_{(\Gamma,v)}$ and ask if we may describe those nonnegative
measures $\mu$ on $(\Gamma,v)^{*}$ such that $H_{(\Gamma,v)}$ is a
bounded map from $\ell^2_v$ to $L^2((\Gamma,v)^{*},\mu)$. This
question is another version of the long-standing problem of finding
criteria akin to the Muckenhoupt $(A_2)$ condition for boundedness
of two-weight Hilbert transforms, cf. the discussion in the last
three chapters of \cite{V}.

The centerpiece of this paper is a solution to the boundedness
problem when $\Gamma$ is exponentially or super-exponentially
``sparse'', i.e., when we have
\begin{equation}
\inf_{n\geq1}|\gamma_{n+1}|/|\gamma_n|>1. \label{expon}
\end{equation}
In this case, $(\Gamma,v)^{*}$ is nonempty and in fact equal to
$\CC\setminus \Gamma$ if and only if \begin{equation}
\sum_{n=1}^\infty \frac{v_n}{1+|\gamma_n|^2} < \infty;
\label{admiss}
\end{equation} we will say that $v$ is an \emph{admissible weight sequence
for $\Gamma$} if \eqref{admiss} holds. When we consider the
boundedness problem for such sparse sequences $\Gamma$, it is quite
natural to partition $\CC$ in the following way. Set
$\Omega_1=\{z\in \CC:\ |z|<\big(|\gamma_1|+|\gamma_2|\big)/2\}$ and
then
\[
\Omega_n=\left\{z\in \mathbb{C} : \
\big(|\gamma_{n-1}|+|\gamma_n|\big)/2\le
|z|<\big(|\gamma_n|+|\gamma_{n+1}|\big)/2\right\}
 \text{ for }  n\ge 2 \label{not2}.
 \]
Our solution to the boundedness problem reads as follows.
\begin{thm} \label{th1}
Suppose that the sequence $\Gamma$ satisfies the sparseness
condition \eqref{expon} and that $v$ is an admissible weight
sequence for $\Gamma$. If $\mu$ is a nonnegative measure on $\CC$
with $\mu(\Gamma)=0$, then the map $H_{(\Gamma,v)}$ is bounded from
$\ell^2_v$ to $L^2(\mathbb{C},\mu)$ if and only if
\begin{equation}\label{trivial} \sup_{n\ge 1} \int_{\Omega_n}
\frac{v_n d\mu(z)}{|z-\gamma_n|^2} <\infty\end{equation} and
\begin{equation}
 \sup_{n\ge 1} \left(\sum_{l=1}^{n} v_l \sum_{m=n+1}^{\infty}\int_{\Omega_m}
 \frac{d\mu(z)}{|z|^{2}}+ \sum_{m=1}^n \mu\left(\Omega_m\right) \sum_{l=n+1}^\infty
 \frac{v_l}{|\gamma_l|^2}
\right) <\infty.
 \label{carl_2}
\end{equation}
\end{thm}
It should be noted that the condition is symmetric in the two
measures $\sum_n v_n \delta_{\gamma_n}$ and $\mu$. This is natural
since the theorem also gives a necessary and sufficient condition
for the adjoint transformation
\[ f\mapsto \left(\int_{\CC}
\frac{f(z) d\mu(z)}{\overline{z}-\overline{\gamma_n}}\right)_n \] to
be bounded from $L^2(\CC,\mu)$ to $\ell^2_v$. The condition
\eqref{carl_2} can be understood as a simple relative to the
classical Muckenhoupt $(A_2)$ condition.

Besides its simplicity, the main virtue of Theorem~\ref{th1} is its
role as a tool in our study of surjectivity and invertibility of
discrete Hilbert transforms. We now turn to the latter topic and
require thus that $\mu$ be a purely atomic measure. In other words,
we are interested in the case when there are a sequence of points
$\Lambda=(\lambda_j)$ in $(\Gamma,v)^{*}$ and a corresponding weight
sequence $w=(w_j)$ such that the discrete Hilbert transform
\begin{equation} (a_n)_n \mapsto \left(\sum_{n=1}^\infty \frac{a_n
v_n}{\lambda_j-\gamma_n} \right)_j \label{discreteHi}
\end{equation}
is bounded from $\ell^2_v$ to $\ell^2_w$. To stress the dependence
on the pair $(\Lambda,w)$, we will denote this transformation by
$H_{(\Gamma,v);(\Lambda,w)}$. A duality argument (see the next
section) shows that if we want $H_{(\Gamma,v);(\Lambda,w)}$ to be a
surjective map as well, then we must have \begin{equation}
\label{admissurj}
w_j=\left[\max_{\|a\|_v=1}|(H_{(\Gamma,v)}a)(\lambda_j)|\right]^{-2}=\left(\sum_{n=1}^\infty
\frac{v_n}{|\lambda_j-\gamma_n|^2}\right)^{-1},
\end{equation}
up to multiplication by a sequence of positive numbers bounded away
from $0$ and $\infty$. When $\Gamma$ and $v$ are given and $\Lambda$
is a sequence in $(\Gamma,v)^*$, we will say that the sequence given
by \eqref{admissurj} is the \emph{Bessel weight sequence for
$\Lambda$ with respect to $(\Gamma,v)$}. The translation into this
discrete setting of Theorem~\ref{th12}, to be stated in
Section~\ref{proofth12} below, is surprisingly subtle: The sequence
$\Lambda$ splits naturally into three subsequences, one that should
be viewed as a perturbation of $\Gamma$ and then two sequences
satisfying only certain ``extreme'' sparseness conditions.

Our next general question is the following: If $w=(w_j)$ is the
Bessel weight sequence for $\Lambda$ with respect to $(\Gamma,v)$
and $H_{(\Gamma,v);(\Lambda,w)}$ is a bounded transformation, is it
possible to split $\Lambda$ into a finite union of subsequences
$\Lambda'$ such that, with $w'$ denoting the subsequence of $w$
corresponding to $\Lambda'$, each of the transformations
$H_{(\Gamma,v);(\Lambda',w')}$ is surjective? As will be explained
in the next section, this question would have a positive answer
should the Feichtinger conjecture hold true. The following result
answers this question when \eqref{expon} holds.
\begin{thm} \label{th12}
Suppose the sequence $\Gamma$ satisfies the sparseness condition
\eqref{expon} and that $v$ is an admissible weight sequence for
$\Gamma$. If $\Lambda$ is a sequence in $\CC\setminus\Gamma$, $w$ is
the Bessel weight sequence for $\Lambda$ with respect to
$(\Gamma,v)$, and the transformation $H_{(\Gamma,v);(\Lambda,w)}$ is
bounded, then $\Lambda$ admits a splitting into a finite union of
subsequences such that, for each subsequence $\Lambda'$ and
corresponding subsequence $w'$ of $w$, the transformation
$H_{(\Gamma,v);(\Lambda',w')}$ is surjective.
\end{thm}
We proceed now to our third main result, which is a general
statement about invertible discrete Hilbert transforms. The
observation that leads to this result, is that the inverse
transformation, if it exists, can be identified effectively as
another discrete Hilbert transform.

To make a precise statement, we introduce the following terminology.
We say that a sequence $\Lambda$ of distinct points in
$(\Gamma,v)^*$ is a \emph{uniqueness sequence for $H_{(\Gamma,v)}$}
if there is no nonzero vector $a$ in $\ell^2_v$ such that
$H_{(\Gamma,v)}a$ vanishes on $\Lambda$; we say that $\Lambda$ is an
\emph{exact uniqueness sequence for $H_{(\Gamma,v)}$} if it is a
uniqueness sequence for $H_{(\Gamma,v)}$, but fails to be so on the
removal of any one of the points in $\Lambda$. If $\Lambda$ is an
exact uniqueness sequence for $H_{(\Gamma,v)}$, then we say that a
nontrivial function $G$ defined on $(\Gamma,v)^*$ is a
\emph{generating function for $\Lambda$} if $G$ vanishes on
$\Lambda$ but, for every $\lambda_j$ in $\Lambda$, there is a
nonzero vector $a^{(j)}$ in $\ell^2_v$ such that
$G(z)=(z-\lambda_j)H_{(\Gamma,v)}a^{(j)}(z)$ for every $z$ in
$(\Gamma,v)^*$. It is clear that if a generating function exists,
then it is unique up to multiplication by a nonzero constant.

We note that if $\Lambda$ is an exact uniqueness sequence for
$H_{(\Gamma,v)}$, then there exists a unique element $e=(e_n)$ in
$\ell^2_v$ such that $H_{(\Gamma,v);(\Lambda,w)}e=(1,0,0,...)$. We
set $\nu=(\nu_n)$ and $\varpi=(\varpi_j)$, where \begin{equation}
\label{nu} \nu_n=v_n |\lambda_1-\gamma_n|^2|e_n|^2,  \end{equation}
$\varpi_1=w_1^{-1}$, and
\begin{equation}\label{varpi}
\varpi_j=w_j^{-1}|\lambda_j-\lambda_1|^{-2} \left| \sum_{n=1}^\infty
\frac{e_n v_n}{(\lambda_j-\gamma_n)^2}\right|^{-2},\end{equation}
presuming the series appearing in the latter expression converges
absolutely. We will see that, plainly, we have absolute convergence
of this series whenever $\Lambda$ admits a generating function.

Our next result reads as follows.

\begin{thm} \label{th13}
Suppose that every exact uniqueness sequence for $H_{(\Gamma,v)}$
admits a generating function. Let $\Lambda$ be a sequence in
$(\Gamma,v)^*$, and let $w$ be the Bessel weight sequence for
$\Lambda$ with respect to $(\Gamma,v)$. If the transformation
$H_{(\Gamma,v);(\Lambda,w)}$ is bounded, then
$H_{(\Gamma,v);(\Lambda,w)}$ is an invertible transformation if and
only if $\Lambda$ is an exact uniqueness sequence for
$H_{(\Gamma,v)}$ and the transformation
$H_{(\Lambda,\varpi);(\Gamma,\nu)}$ is bounded.
\end{thm}

Note that when we write `$H_{(\Lambda,\varpi);(\Gamma,\nu)}$ is
bounded', it is implicitly understood that $\Gamma\subset (\Lambda,
\varpi)^*$.

We may observe that if $\gamma_n \to\infty$ when $n\to\infty$, then
the function
\begin{equation}\label{Phi} \Phi(z)=(z-\lambda_1)\sum_{n=1}^\infty
\frac{e_n v_n}{z-\gamma_n},
\end{equation}
and its reciprocal $\Psi=1/\Phi$ are meromorphic functions in $\CC$,
and $\Phi$ is then the generating function for $\Lambda$. We may
then rewrite the expressions for $\nu$ and $\varpi$ as
\begin{equation}\label{PsiPhi} \nu_n=v_n |\Psi'(\gamma_n)|^{-2} \ \ \ \text{and} \
\ \ \varpi_j=w_j^{-1} |\Phi'(\lambda_j)|^{-2}.\end{equation}

Combining Theorem~\ref{th13} with Theorem~\ref{th1}, we will obtain
computable and geometric invertibility criteria when $\Gamma$ is a
sparse sequence as defined by \eqref{expon}. To illustrate the
nature of these criteria, we highlight the following concrete
example, where it is tacitly assumed that \eqref{expon} holds and
that $w=(w_n)$ is the Bessel weight sequence for $\Lambda$ with
respect to $(\Gamma,v)$.

\begin{example} Set $V_n=\sum_{m=1}^{n-1}v_m$ for $n>1$. Assume that
both $v_n=o(V_n)$ and $V_n\to \infty$ when $n\to \infty$, and write
$\Gamma=(\gamma_n)$ and $\Lambda=(\lambda_n)$, with both sequences
indexed by the positive integers. Moreover, assume that there exists
a positive constant $C$ such that
\begin{equation} \label{besselex}
\frac{|\gamma_n-\lambda_n|}{|\gamma_n|} \le C\, \frac{v_n}{V_n}
\end{equation}
for every positive integer $n$.
\begin{itemize}
\item[(0)] If, in addition, there is a real
constant $c<1/2$ such that
\[
\frac{|\gamma_n|}{|\lambda_n|} - 1 \le c \, \frac{v_n}{V_n}
\]
for all sufficiently large $n$, then $H_{(\Gamma,v);(\Lambda,w)}$ is
an invertible transformation.
\item[(1)] If, on the other hand, there is a positive
constant $c >1/2$ such that
\[
\frac{|\gamma_n|}{|\lambda_n|} - 1 \ge c \, \frac{v_n}{V_n}
\]
for all sufficiently large $n$, then
$H_{(\Gamma,v);(\Lambda^{(1)},w^{(1)})}$ is an invertible
transformation, where $\Lambda^{(1)}=(\lambda_2, \lambda_3,...)$ and
$w^{(1)}=(w_2,w_3,...)$.
\end{itemize}
\end{example}

It follows from \eqref{besselex} that $H_{(\Gamma,v);(\Lambda,w)}$
is a bounded transformation, while the respective conditions in (0)
and (1) imply that the inverse transformations are bounded, subject
to the proviso that, when (1) holds, one point be removed from
$\Lambda$. This rather puzzling example can be seen as an analogue
of the Kadets $1/4$ theorem for complex exponentials \cite{Kad}, cf.
the discussion in the next section. We note that if we have the
precise relation
\[
\frac{|\gamma_n|}{|\lambda_n|} - 1 = \half\, \frac{v_n}{V_n},
\]
then neither $H_{(\Gamma,v);(\Lambda,w)}$ nor
$H_{(\Gamma,v);(\Lambda^{(1)},w^{(1)})}$ is an invertible
transformation. Another curious point is that if we replace the
condition that $V_n\to \infty$ by the assumption that $\sup_n
V_n<\infty$, then \eqref{besselex} automatically implies that
$H_{(\Gamma,v);(\Lambda,w)}$ is an invertible transformation.

An interesting feature of our results for sparse sequences is that
invertibility implies that $\Lambda$ is a perturbation of $\Gamma$,
in a sense to be made precise. As a consequence, we will see that
there may exist bounded transformations $H_{(\Gamma,v);(\Lambda,w)}$
such that no infinite subsequence $\Lambda'$ of $\Lambda$ is also a
subsequence of another sequence $\Lambda''$ for which the associated
Hilbert transform is invertible.

Before we turn to the proofs of our three main theorems and a
detailed discussion of further results as alluded to above, we will
place our study in context by explaining in detail how to translate
our problems and findings into statements about systems of
reproducing kernels. We are particularly interested in such systems
in the distinguished case when the underlying Hilbert space is a de
Branges spaces or, more generally, a model subspace of $H^2$. The
study of such systems has a long history, beginning with the work of
Paley and Wiener on systems of nonharmonic Fourier series, and is
related to a number of interesting applications. More recently, the
advent of the Feichtinger conjecture has given additional impetus to
the subject; the special case of the conjecture pertaining to
discrete Hilbert transforms appears as an interesting setting in
which the ramifications of the general Feichtinger conjecture could
be explored.

The present investigation originated in questions raised about
systems of reproducing kernels. During the course of our work, we
have found it both useful and conceptually appealing to transform
the subject into a study of the mapping properties of discrete
Hilbert transforms. We have learned to appreciate that the essential
difficulties thus seem to appear in a more succinct form.

We close this introduction with a few words on notation. Throughout
this paper, the notation $U(z)\lesssim V(z)$ (or equivalently
$V(z)\gtrsim U(z)$) means that there is a constant $C$ such that
$U(z)\leq CV(z)$ holds for all $z$ in the set in question, which may
be a Hilbert space, a set of complex numbers, or a suitable index
set. We write $U(z)\simeq V(z)$ if both $U(z)\lesssim V(z)$ and
$V(z)\lesssim U(z)$. As above, a sequence $\Gamma=(\gamma_n)$ of
distinct complex numbers will frequently be viewed as a subset of
$\CC$. Sometimes we will need to remove a single point, say
$\gamma_m$, from such a sequence. The sequence thus obtained will
then be written $\Gamma\setminus \{\gamma_m\}$, where $\{\gamma_m\}$
denotes the set consisting of the single point $\gamma_m$.

\section{Translation into statements about systems of reproducing kernels}


\subsection{A class of Hilbert spaces}\label{Hspace}

Let $\Hg$ be a Hilbert space of complex-valued functions defined on
some set $\Omega$ in $\CC$. We will say that a sequence $\Lambda$ of
distinct points in $\Omega$ is a \emph{uniqueness sequence for}
$\Hg$ if no nonzero function in $\Hg$ vanishes on $\Lambda$; we say
that $\Lambda$ is an \emph{exact uniqueness sequence for} $\Hg$ if
it is a  \emph{uniqueness sequence for} $\Hg$, but fails to be so on
the removal of any one of the points in $\Lambda$.  If $\Lambda$ is
an exact uniqueness sequence for $\Hg$, then we say that a
nontrivial function $G$ defined on $\Omega$ is a \emph{generating
function for $\Lambda$} if $G$ vanishes on $\Lambda$ but, for every
$\lambda_j$ in $\Lambda$, there is a nonzero function $g_j$ in $\Hg$
such that $G(z)=(z-\lambda_j)g_j(z)$ for every $z$ in $\Omega$. It
is clear that if a generating function exists, it is unique up to
multiplication by a nonzero constant.

We will assume that $\Hg$ satisfies the following three axioms:
\begin{itemize}
\item[(A1)] $\Hg$ has a
reproducing kernel $\kappa_{\lambda}$ at every point $\lambda$ in
$\Omega$, i.e., the point evaluation functional $\kappa_\lambda:
f\rightarrow f(\lambda)$ is continuous in $\Hg$ for every $\lambda$
in $\Omega$.
\item[(A2)]Every exact uniqueness sequence for $\Hg$ admits a generating function.
\item[(A3)] There exists a sequence of distinct points $\Gamma=(\gamma_n)$ in
$\Omega$ such that the sequence of normalized reproducing kernels
$\big(\kappa_{\gamma_n}/\|\kappa_{\gamma_n}\|_{\Hg}\big)$
constitutes a Riesz basis for $\Hg$. In addition, there is at least
one point $z$ in $\Omega\setminus \Gamma$ for which $\kappa_{z}\neq
0$.
\end{itemize}

The second axiom (A2) may be viewed as a weak statement about the
possibility of dividing out zeros. To see this, we may observe that
(A2) holds trivially if $\Hg$ has the property that whenever
$f(\lambda)=0$ for some $f$ in $\Hg$ and $\lambda$ in $\Omega$, we
have that $f(z)/(z-\lambda)$ also belongs to $\Hg$. On the other
hand, (A2) and (A3) lead to a representation of functions in $\Hg$
(see below) which shows that if $\lambda$ is a point in
$\Omega\setminus\overline{\Gamma}$ such that $\kappa_\lambda\neq 0$,
then $f(z)/(z-\lambda)$ is in $\Hg$ whenever $f$ is in $\Hg$ and
$f(\lambda)=0$. In general, however, this division property need not
hold at the accumulation points of $\Gamma$ when we only assume
(A2).

The Riesz basis $(\kappa_{\gamma_n}/\|\kappa_{\gamma_n}\|_{\Hg})$
has a biorthogonal basis, which we will call $(g_n)$. By axiom (A2),
we may write $G(z)=c_n (z-\gamma_n)g_n(z)$ for some nonzero constant
$c_n$. We use the suggestive notation $G'(\gamma_n)$ for the value
of $G(z)/(z-\gamma_n)$ at $\gamma_n$. The sequence $g_n$ is also a
Riesz basis for $\Hg$, and therefore every vector $h$ in $\Hg$ can
be written as
\begin{equation} h(z)=\sum_{n}
h(\gamma_n)\frac{G(z)}{G'(\gamma_n)(z-\gamma_n)},
\label{lagrange}\end{equation} where the sum converges with respect
to the norm of $\Hg$ and
\[ \|h\|_{\Hg}^2\simeq \sum_n \frac{|h(\gamma_n)|^2}{\|\kappa_{\gamma_n}\|_{\Hg}^2}
<\infty; \] since point evaluation at every point $z$ is a bounded
linear functional, \eqref{lagrange} also converges pointwise in
$\Omega\setminus\Gamma$. By the assumption that $h\mapsto
(h(\gamma_n)/\|\kappa_{\gamma_n}\|_{\Hg})$ is a bijective map from
$\Hg$ to $\ell^2$, this means that \begin{equation} \sum_n
\frac{\|\kappa_{\gamma_n}\|_{\Hg}^2}{|G'(\gamma_n)|^2|z-\gamma_n|^2}<\infty
\label{pointwise}
\end{equation}
whenever $z$ is in $\Omega\setminus\Gamma$. By the last part of
axiom (A3), there is at least one such $z$ in $\Omega\setminus
\Gamma$. Therefore, \eqref{pointwise} implies that
\begin{equation} \sum_{n} \frac{v_n}{1+|\gamma_n|^2}<\infty.
\label{adm}
\end{equation}

We may now change our viewpoint: Given a sequence of distinct
complex numbers $\Gamma=(\gamma_n)$ and a weight sequence $v=(v_n)$
that satisfy the admissibility condition \eqref{adm}, we introduce
the space $\Hg(\Gamma,v)$ consisting of all functions
\[ f(z)=\sum_{n=1}^\infty\frac{a_n v_n}{z-\gamma_n} \]
for which
\[ \|f\|_{\Hg(\Gamma,v)}^2=\sum_{n=1}^\infty |a_n|^2v_n <\infty,\]
assuming that the set $(\Gamma,v)^*$ is nonempty. Thus we obtain the
value of a function $f$ in $\Hg(\Gamma, v)$ at a point $z$ in
$(\Gamma, v)^{*}$ by computing a discrete Hilbert transform.

We say that a nonnegative measure $\mu$ on $(\Gamma,v)^{*}$ is a
\emph{Carleson measure for} $\mathscr{H}(\Gamma,v)$ if there is a
positive constant $C$ such that
\[\int_{(\Gamma,v)^*}|f(z)|^2d\mu(z)\le C \|f\|^2_\mathscr{H}\]
holds for every $f$ in $\mathscr{H}(\Gamma,v)$. It is now immediate
that $\mu$ is a Carleson measure for $\mathscr{H}(\Gamma,v)$ if and
only if the map $H_{(\Gamma,v)}$ is bounded from $\ell^2_v$ to
$L^2((\Gamma,v)^*,\mu)$.

\subsection{Systems of reproducing kernels in
$\mathscr{H}(\Gamma,v)$}\label{systems}

We begin again with some generalities. Let $\mathcal{H}$ be a
complex Hilbert space and $(f_j)$ a sequence of unit vectors in
$\mathcal{H}$. We say that $(f_j)$ is a \emph{Bessel sequence in
$\mathcal{H}$} if there is a positive constant $C$ such that the
inequality
$$ \sum_j |\langle f, f_j\rangle_{\mathcal{H}}|^2 \leq C
\|f\|_{\mathcal{H}}^2$$ holds for every $f$ in $\mathcal{H}$. The
sequence $(f_j)$ is a \emph{Riesz basic sequence in $\mathcal{H}$}
if there are positive constants $A$ and $B$ such that the
inequalities
\begin{equation} \label{riesz} A \sum_j |c_j|^2 \leq \Big \| \sum_j
c_j f_j\Big \|_{\mathcal{H}}^2 \leq B \sum_n |c_n|^2 \end{equation}
hold for every finite sequence of scalars $(c_j).$ The Feichtinger
conjecture claims that \emph{every Bessel sequence of unit vectors
can be expressed as a finite union of Riesz basic sequences.}

In an intriguing series of papers \cite{PG1, PG2, PG4, PG5},  it has
been revealed that the Feichtinger conjecture is equivalent to the
long-standing Kadison--Singer conjecture \cite{KS}. We refer to the
recent paper \cite{PG3} for a historical account and an interesting
reformulation of the Feichtinger conjecture.

Before turning to the special case of normalized reproducing kernels
for $\Hg(\Gamma, v)$, we record the following consequence of the
open mapping theorem \cite[p. 73]{WR}.

\begin{lemma} \label{duality}
Suppose $T$ is a bounded linear transformation from a Hilbert space
${\mathcal H}_1$ to another Hilbert space ${\mathcal H}_2$. Then $T$
is surjective if and only if the adjoint transformation $T^*$ is
bounded below.
\end{lemma}

If we let $T$ be the map $f\mapsto (\langle f,f_j\rangle_{\mathcal
H})$, then $T^*(c_j)=\sum_j c_j f_j$. Thus it follows from
Lemma~\ref{duality} that $(f_j)$ is a Riesz basic sequence if and
only if it is a Bessel sequence for which the moment problem
$\langle f, f_j\rangle_{\mathcal{H}} = a_j$ has a solution $f$ in
$\mathcal{H}$ for every square-summable sequence $(a_j).$ We may
also set $T=H_{(\Gamma,v);(\Lambda,w)}$ and observe that
Lemma~\ref{duality} gives the necessary condition
\[
w_j\simeq \left(\sum_{n=1}^\infty
\frac{v_n}{|\lambda_j-\gamma_n|^2}\right)^{-1}\] for surjectivity of
the transformation $H_{(\Gamma,v);(\Lambda,w)}$.

We return now to the space $\Hg(\Gamma, v)$. We note that the
reproducing kernel of $\Hg(\Gamma, w)$ at a point $z$ in
$(\Gamma,v)^{*}$ is
\[ k_{z}(\zeta)=\sum_{n=1}^\infty
\frac{v_n}{(\overline{z}-\overline{\gamma_n})(\zeta-\gamma_n)}; \]
this is a direct consequence of the definition of $\Hg(\Gamma, v)$.
Given a sequence $\Lambda=(\lambda_j)$ in $(\Gamma,v)^*$, we
associate with it the corresponding sequence of normalized
reproducing kernels
$\left(k_{\lambda_j}/\|k_{\lambda_j}\|_{\Hg(\Gamma,v)}\right)$. We
observe that if $w$ is the Bessel weight sequence for $\Lambda$ with
respect to $(\Gamma,v)$, then the transformation
$H_{(\Gamma,v);(\Lambda,w)}$ is bounded if and only if the system
$(k_{\lambda_j}/\|k_{\lambda_j}\|_{\Hg(\Gamma,v)})$ is a Bessel
sequence in $\Hg(\Gamma,v)$. Moreover, this transformation is both
bounded and surjective if and only if the system
$(k_{\lambda_j}/\|k_{\lambda_j}\|_{\Hg(\Gamma,v)})$ is a Riesz basic
sequence in $\Hg(\Gamma,v)$. Thus Theorem~\ref{th12} shows that the
Feichtinger conjecture holds true for systems of normalized
reproducing kernels in $\Hg(\Gamma,v)$ whenever $\Gamma$ is a
sequence satisfying the sparseness condition \eqref{expon}.

We finally note that if $w$ is the Bessel weight sequence for
$\Lambda$ with respect to $(\Gamma,v)$, then  the transformation
$H_{(\Gamma,v);(\Lambda,w)}$ is invertible if and only if the system
$(k_{\lambda_j}/\|k_{\lambda_j}\|_{\Hg(\Gamma,v)})$ is a Riesz basis
for $\Hg(\Gamma,v)$. If $\Hg(\Gamma,v)$ is obtained from a space
$\Hg$ satisfying (A1), (A2), (A3), as described in the previous
subsection, then Theorem~\ref{th13} applies. In the special case
when $\gamma_n\to\infty$ as $n\to\infty$, we may write the
meromorphic function defined in \eqref{Phi} as $\Phi=F/G$, with $G$
again denoting the generating function for $\Gamma$ and $F$ an
entire function with a simple zero at each point $\lambda_j$. Then
the expressions appearing in \eqref{PsiPhi} can be restated as
\[\nu_n=v_n |F(\gamma_n)/G'(\gamma_n)|^{2} \ \ \ \text{and} \
\ \ \varpi_j=w_j^{-1} |G(\lambda_j)/F'(\lambda_j)|^2.\]

\subsection{First example: de Branges spaces}\label{example1}

The prime examples of Hilbert spaces belonging to the general class
described in subsection~\ref{Hspace} are found among so-called de
Branges spaces and model subspaces of $H^2$. To begin with, we note
that de Branges spaces may be defined in terms of axioms that are
very similar to those introduced above. Indeed, a Hilbert space
$\Hg$ of entire functions which contains a non-zero element is
called a \emph{de Branges space} if it satisfies the following three
axioms:
\begin{itemize}
\item[(H1)] $\Hg$ has a reproducing kernel $\kappa_{\lambda}$ at
every point $\lambda$ in $\CC$, i.e., the point evaluation
functional $\kappa_\lambda: f\rightarrow f(\lambda)$ is continuous
in $\Hg$ for every $\lambda$ in $\CC$.
\item[(H2)] If $f$ is in $\Hg$ and $f(\lambda)=0$ for some point $\lambda$ in $\CC$,
then $f(z)(z-\bar{\lambda})/(z-\lambda)$ is in $\Hg$ and has the
same norm as $f$.
\item[(H3)] The function $\overline{f(\bar{z})}$ belongs to $\Hg$
whenever $f$ belongs to $\Hg$, and it has the same norm as $f$.
\end{itemize}

The general reference for de Branges spaces is the book \cite{DB}.
The leading example of a de Branges space is the Paley--Wiener
space, which consists of those entire functions of exponential type
at most $\pi$ that are square summable when restricted to the real
line.

A space $\Hg$ that satisfies (H1), (H2), (H3), will in particular
satisfy (A1), (A2), (A3) with $\Omega=\CC$. Indeed, we observe that
then (H1) and (A1) coincide, and it is also plain that (H2) implies
(A2). One of the basic results in de Branges's theory is that a
space that satisfies (H1), (H2), (H3), will have an orthogonal basis
consisting of reproducing kernels $\kappa_{\gamma_n}$ with
$\Gamma=(\gamma_n)$ being a sequence of real points. Thus, in
particular, (H1), (H2), (H3) imply that our third general axiom (A3)
holds. In the case of the Paley--Wiener space, we have an orthogonal
basis of reproducing kernels associated with the sequence of
integers, leading to what is known as the cardinal series or the
Shannon sampling theorem.

Another way of defining de Branges spaces is as follows. We say that
an entire function $E$ belongs to the Hermite--Biehler class if it
has no real zeros and satisfies $|E(z)|>|E(\overline{z})|$ for $z$
in the upper half-plane. For a given function $E$  in the
Hermite--Biehler class, we let $H(E)$ be the Hilbert space
consisting of all entire functions $f$ such that both $f/E$ and
$f^*/E$ belong to $H^2$. Here $f^*(z)= \overline{f(\overline{z})}$
and $H^2$ is the Hardy space of the upper half-plane, viewed in the
usual way as a subspace of $L^2({\Bbb R})$. We set
\[ \|f\|_{H(E)}^2=\int_{-\infty}^{\infty}\frac{|f(x)|^2}{|E(x)|^2}
dx.\] Then $H(E)$ is a de Branges space, and every de Branges space
can be obtained in this way via a function $E$ in the
Hermite--Biehler class. We arrive at the Paley--Wiener space by
setting $E(z)=e^{-i\pi z}$.

It follows from the preceding remarks that Theorem~\ref{th1} and
Theorem~\ref{th12} apply to de Branges spaces with orthogonal bases
of reproducing kernels located at a sequence of nonzero real points
$\gamma_n$ such that $\inf_n |\gamma_{n+1}|/|\gamma_n|>1$.

\subsection{Second example: Model subspaces of
$H^2$} Given an inner function $I$ in the upper half-plane, we
define the \emph{model subspace $K^2_I$} as
\[K^2_I = H^2\ominus IH^2;\]
which is the orthogonal complement in $H^2$ of functions divisible
by the inner function $I$. These spaces are, by a classical theorem
of A. Beurling \cite{BB}, the subspaces of $H^2$ that are invariant
with respect to the backward shift. We refer to \cite{NK,NK1,NK2}
for information about the model theory related to the backward
shift. The elements of $K^2_I$ (originally defined in
$\mathbb{C}^+$) have meromorphic extensions into $\mathbb{C}$  if
the function $I$ has such an extension. In this case, we have the
relation $I=E^{*}/E$ and the map $f\mapsto f/E$ is a unitary map
from $H(E)$ to $K^2_I$. Thus de Branges spaces can be viewed as a
subclass of the collection of all model subspaces of $H^2$.

We now prove that every model subspace satisfies axiom (A2) from
Subsection~\ref{Hspace}. This is obvious if we consider $K^2_I$ as a
space of functions on the upper half-plane, but for our purposes it
is essential that we also include those points on the real line at
which point evaluation makes sense. We will need the fact that the
reproducing kernel for $K_I^2$ at some point $\zeta$ in the upper
half-plane is
\[ \kappa_{\zeta}(z)=\frac{i}{2\pi}\cdot\frac{1-\overline{I(\zeta)}
I(z)}{z-\overline{\zeta}}. \] This formula extends to each point on
the real line at which every function in $K^2_I$ has a nontangential
limit whose modulus is bounded by a constant times the $H^2$ norm of
the function. A paper  of P. Ahern and D. Clark \cite{AC} gives that
these are exactly the points $x$ at which $I'$ has a nontangential
limit.

\begin{lemma}\label{axA2}
The Hilbert space $K^2_I$, viewed as a space of functions on the set
\[\Omega=\left\{z=x+iy: \ y\ge 0 \ \text{and}\ f\mapsto f(z) \ \text{is bounded}
\right\},\]  satisfies axiom (A2) of Subsection~\ref{Hspace}.
\end{lemma}

To make the proof more transparent, we single out the main technical
ingredient as a separate lemma.

\begin{lemma}\label{contin}
If $x_0$ is a point on the real line at which the point evaluation
functional for $K_I^2$ is bounded, then
$\|\kappa_{x_0+iy}-\kappa_{x_0}\|_{H^2}\to 0$ when $y\to 0$.
\end{lemma}
\begin{proof}
Assuming $I(x_0)=1$, we may write
\[ \begin{split}
\frac{2\pi}{i}(\kappa_{x_0+iy}(t)-\kappa_{x_0}(t))&  =
\frac{1-\overline{I(x_0+iy)}I(t)}{t-(x_0+iy)}-\frac{1-I(t)}{t-x_0} \\
& =  \frac{(1-\overline{I(x_0+iy)}) I(t)}{t-(x_0+iy)}-
\frac{(1-I(t))iy t}{(t-x_0)(t-(x_0+iy))}.
\end{split} \]
Here the first term on the right-hand side has $H^2$ norm bounded by
a constant times $y^{\half}$ in view of the theorem of Ahern and
Clark \cite{AC}, while the second term tends to $0$ when $y\to 0$,
by Lebesgue's dominated convergence theorem.
\end{proof}
\begin{proof}[Proof of Lemma~\ref{axA2}]
Let $\Lambda$ be an exact uniqueness set for $K^2_I$ consisting of
points in $\Omega$. We will let $g_{j}$ denote the unique function
in $K^2_I$ such that $g_j(\lambda_l)$ equals $0$ when $l\neq j$ and
$1$ for $l=j$. We can choose an arbitrary point in $\Lambda$, say
$\lambda_1$, and choose $G(z)=(z-\lambda_1)g_1(z)$ as our candidate
for a generating function. It is plain that if $\lambda_j$ is a
point in the open half-plane, then
$g_j(z)=G(z)/[G'(\lambda_j)(z-\lambda_j)]$. The difficulty occurs if
$\lambda_j$ is a point on the real line. In this case, if we replace
$\lambda_j$ by $\lambda_j+i\varepsilon$, then the modified sequence
$\Lambda^{(\varepsilon)}$ will still be an exact uniqueness sequence
for $K^2_I$. In fact, by Lemma~\ref{contin}, the function $g_1$
vanishing on $\Lambda^{(\varepsilon)}\setminus\{\lambda_1\}$ will
vary continuously with $\varepsilon$. Thus the corresponding
generating function $G_{\varepsilon}(z)$ will tend to $G(z)$ for
every point in the upper half-plane when $\varepsilon \to 0$. On the
other hand, another application of Lemma~\ref{contin} gives that
$G_{\varepsilon}(z)/[G'_{\varepsilon}(\lambda_j+i\varepsilon)(z-\lambda_j+i\varepsilon)]\to
g_j(z)$ in $K^2_I$ when $\varepsilon \to 0$. Lemma~\ref{contin} also
gives that $G'_{\varepsilon}(\lambda_j+i\varepsilon)$ converges to a
finite number, say $1/\alpha$, and we may therefore conclude that
$g_j(z)=\alpha G(z)/(z-\lambda_j)$.
\end{proof}

As for axiom (A3), it remains an open problem to decide whether
every model subspace $K^2_I$ has a Riesz basis of normalized
reproducing kernels. Thus it is not known whether the class of
spaces introduced in Subsection~\ref{Hspace} includes all model
subspaces. However, there exists an interesting class of model
subspaces that actually possess orthogonal bases of reproducing
kernels associated with sequences of real points. Such bases, to be
discussed briefly below, are called Clark bases \cite{CL}. We also
note that if the inner function $I$ happens to be an interpolating
Blaschke product, then it is well-known and easy to show that
$K^2_I$ has a Riesz basis of normalized reproducing kernels
associated with the sequence of zeros of $I$.

The spaces $K^2_I$ that possess Clark bases, correspond precisely to
those spaces $\Hg(\Gamma, v)$ for which $\Gamma$ is a real sequence.
To get from $\Hg(\Gamma, v)$ to the corresponding space $K_I^2$, we
construct the Herglotz function \begin{equation}
\label{phi}\varphi(z)=\sum_{n=1}^\infty v_n
\left(\frac{1}{\gamma_n-z}-\frac{\gamma_n}{1+\gamma_n^2}\right).
\end{equation} Then
\begin{equation} \label{inner}
I(z)=\frac{\varphi(z)-i}{\varphi(z)+i} \end{equation} will be an
inner function in the upper half-plane, and the map $f\mapsto
(1-I)f$ will be a unitary map from $\Hg(\Gamma, v)$ to $K^2_I$; it
is implicit in this construction that in fact every function in
$K^2_I$ has a nontangential limit at each point $\gamma_n$ and also
that the corresponding point evaluation functional is bounded at
$\gamma_n$. We refer to our recent paper \cite{BMS}, where Clark
bases are treated in more detail and it is shown that families of
orthogonal bases of reproducing kernels can only exist in spaces of
the form $\Hg(\Gamma,v)$ when $\Gamma$ is a subset of a straight
line or a circle.

We conclude that our general discussion (including
Theorem~\ref{th13}) applies to model subspaces $K^2_I$ that possess
 Clark bases.

\subsection{Carleson measures and systems of reproducing kernels in
de Branges spaces and model subspaces of $H^2$}\label{deBmod} A
long-standing problem in the function theory of de Branges spaces
and, more generally, of model subspaces $K_I^2$ is to describe their
Carleson measures. Only a few special cases have been completely
understood. One such case is when $I$ is a so-called one-component
inner function, i.e., when there exists a positive number $
\varepsilon$ with $0 < \varepsilon < 1$ such that the set of points
$z$ in the upper half-plane satisfying $|I(z)| < 1-\varepsilon$ is
connected \cite{WC, TV}. Other partial results can be found in
\cite{Al, Al1, Ba1, WC1, DY1, TV}. The lack of a general result on
Carleson measures and, more specifically, on the geometry of Bessel
sequences of normalized reproducing kernels is an obvious challenge
when we address the Feichtinger conjecture in this setting.

Since positive results on Carleson measures are scarce, we mention
without proof the following observation: A suitable adaption of
Theorem~\ref{th1} gives a description of any Carleson measure $\mu$
restricted to a cone $\{z=x+iy: \ |z-x_0|< Cy\}$; here $x_0$ is an
arbitrary real point and $C$ a positive constant.  To arrive at this
result, one may represent the space by means of its Clark basis or
more generally as an $L^2$ space with respect to a Clark measure
\cite{CL}, and act similarly as in Section~\ref{proofth12}.

A conjecture of W. Cohn \cite{WC}, suggesting that it would suffice
to verify the Carleson measure condition for reproducing kernels,
was refuted by F. Nazarov and A. Volberg \cite{NV}. In the present
paper, we will likewise give a negative answer to the following
question raised by A. Baranov:
\begin{question}
 Suppose that $\Gamma$ and
$\Lambda$ are disjoint sequences of real numbers and that
$\gamma_n\nearrow \infty$. If $w$ is the Bessel weight sequence for
$\Lambda$ with respect to
 $(\Gamma,v)$ and $H_{(\Gamma,v);(\Lambda,w)}$ is a bounded transformation,
then is it true that there is a uniform bound on the number of
$\lambda_j$ found between two points $\gamma_n$ and $\gamma_{n+1}$?
\end{question}
A slight modification of our general approach will lead to a
suitable example, to be presented in Subsection~\ref{cluster}, with
no such uniform bound.

It was recently shown by A. Baranov and K. Dyakonov \cite{BD} that
the Feichtinger conjecture holds true for normalized reproducing
kernels for $K^2_I$ when either $I$ is a one-component inner
function or the points $\lambda_j$ associated with the sequence of
reproducing kernels $\kappa_{\lambda_j}$ satisfy
\begin{equation} \label{fromone}
\sup_j|I(\lambda_j)|<1.
\end{equation}
In the latter case, a complete description of Riesz basic sequences
exists \cite{HNP}, and this result plays an essential role in the
proof. Baranov and Dyakonov used their result for the case when
\eqref{fromone} holds to treat the general case of one-component
inner functions. Their approach was to split the half-plane into two
regions, one in which $|I(z)|$ is bounded away from $1$ and another
in which a perturbation argument for Clark bases applies.

Our results complement those of \cite{BD} and show that, in general,
one needs additional ideas to resolve the Feichtinger conjecture for
normalized reproducing kernels for model subspaces. First of all,
since $\Gamma$ satisfies the sparseness condition \eqref{expon}, the
corresponding inner function $I$, defined by \eqref{inner} via
\eqref{phi}, is not a one-component inner function. Moreover, as
will be revealed in Subsection~\ref{geometry}, there exist inner
functions $I$ and Bessel sequences of normalized reproducing kernels
$\kappa_{\lambda_j}/\| \kappa_{\lambda_j}\|_2$ for $K_I^2$ such that
$|I(\lambda_j)|\to 1$ but no infinite subsequence is a subsequence
of a Riesz basis.

By the observation made at the end of the previous subsection, the
problem of describing all Riesz bases of normalized reproducing
kernels for $K^2_I$ is part of the problem of deciding when discrete
Hilbert transforms $H_{(\Gamma,v);(\Lambda,w)}$ are bounded. The
most far-reaching result known about such bases is that found in
\cite{HNP} dealing with the case when \eqref{fromone} holds. The
general result in \cite{HNP} for this particular case leads to a
description of all Riesz bases of normalized reproducing kernels for
the Paley--Wiener space and also for a wider class of de Branges
spaces known as weighted Paley--Wiener spaces \cite{LS2}. One of the
main points of \cite{HNP} is that when \eqref{fromone} holds, one
can transform the problem into a question about invertibility of
Toeplitz operators and then apply the Devinatz--Widom theorem.
Another approach, closer in spirit to the present work, can be found
in \cite{LS1}, where the Riesz basis problem is explicitly related
to a boundedness problem for Hilbert transforms.

\subsection{Third example: ``Small'' Fock-type spaces}

It may be noted that our work gives a full description of the
Carleson measures and the Riesz bases of normalized reproducing
kernels for certain ``small'' Fock-type spaces studied recently by
A. Borichev and Yu. Lyubarskii \cite{BL}. The spaces $\Hg$
considered by these authors consist of all entire functions $f$ such
that
\[ \|f\|_{\varphi}^2 =\int_{\CC} |f(z)|^2 e^{-2\varphi(|z|)} dm(z) <\infty,\]
where $\varphi$ is a positive, increasing, and unbounded function on
$[0,\infty)$ and $m$ denotes Lebesgue area measure on $\CC$. The
main point of \cite{BL} is that if $\varphi$ grows ``at most as
fast'' as $[\log(1+r)]^2$, then the corresponding space $\Hg$ has a
Riesz basis of reproducing kernels and, conversely, if the growth of
$\varphi$ is ``faster'' than $[\log(1+r)]^2$, then no such basis
exists. It is proved that when $\varphi(r)=[\log(1+r)]^2$, we can
choose such a basis associated with a sequence $\Gamma=(\gamma_n)$
satisfying $|\gamma_n|=e^{n/2}$; if
$\varphi(r)=[\log(r+1)]^{\alpha}$ with $1<\alpha<2$, then the growth
of $|\gamma_n|$ will be super-exponential. The detailed
interpretation of our results in this setting will be presented in
the forthcoming doctoral dissertation \cite{Meng}.

In view of the observations in the previous subsections, the results
of Borichev and Lyubarskii clarify when a Fock-type space equals a
de Branges space, i.e., when the two spaces consist of the same
entire functions and have equivalent norms.

\section{The boundedness problem}\label{proofth12}

\subsection{Proof of Theorem~\ref{th1}}
In what follows, we will use the notation
\begin{equation}
 V_1=1,\ \ \  V_n=\sum_{j=1}^{n-1} v_j, \ \ \ \ \ \text{and} \ \ \ \
\ P_n=\sum_{j=n+1}^\infty \frac{v_j}{|\gamma_j|^2}. \label{not1}
\end{equation}
\begin{proof}[Proof of the necessity of the conditions in Theorem~\ref{th1}]
We observe first that the necessity of \eqref{trivial} is obvious:
Just apply $H_{(\Gamma,v)}$ to the sequence $e^{(n)}=(e_m^{(n)})$
with $e^{(n)}_n=1$ and $e_m^{(n)}=0$ for $m\neq n$.

To show that \eqref{carl_2} is also a necessary condition, we begin
by looking at the sequence $c^{(n)}=(c_m^{(n)})$ so that
$c_m^{(n)}=1$ for $m < n$ and $c_m^{(n)}=0$ otherwise. We observe
that $\|c^{(n)}\|_{v}^2=V_n$ and note that for $z$ in $\Omega_l$ and
$l\ge n$ we have
\[
|H_{(\Gamma,v)}c^{(n)}(z)|^2=\biggl{|}\sum_{m=1}^{n-1}
\frac{v_m}{z-\gamma_m}\biggr{|}^2\gtrsim \frac{V_n^2}{ |z|^{2}}.
\]
Taking into account the boundedness of $H_{(\Gamma,v)}$, we deduce
from this that
 \[
V_n \gtrsim \int_{{\mathbb{C}}} |H_{(\Gamma,v)}c^{(n)}(z)|^2 d\mu(z)
 \gtrsim V_n^2 \sum_{m\ge n}
\int_{\Omega_m} \frac{d\mu(z)}{|z|^2}.\]

On the other hand, if we set $a^{(n)}=(a_m^{(n)})$ so that
$a_m^{(n)}=1/\overline{\gamma_m}$ for $m>n$ and $a_m^{(n)}=0$
otherwise, then $\|a^{(n)}\|_{v}^2=P_n$. We note that for $z$ in
$\Omega_l$ and $l\le n$ we have
\[ |H_{(\Gamma,v)}a^{(n)}(z)|^2=\biggl{|}\sum_{m=n+1}^\infty
\frac{v_m}{\overline{\gamma_m}(z-\gamma_m)}\biggr{|}^2\gtrsim P_n^2
.\] Thus
 \[
P_n \gtrsim \int_{{\mathbb{C}}} |H_{(\Gamma,v)}a^{(n)}(z)|^2 d\mu(z)
 \gtrsim P_n^2 \sum_{m\le n}
\mu(\Omega_m).\]
\end{proof}

\begin{proof}[Proof of the sufficiency of the conditions in Theorem~\ref{th1}]
Let $a=(a_n)$ be an arbitrary sequence in $\ell_v^2$. We make first
the following estimate:
\begin{eqnarray*}
\label{4} \int_{\Omega_n} |H_{(\Gamma,v)}a(z)|^2 d\mu(z) \leq 3
\int_{\Omega_n}\bigg[ \bigg|\sum_{m=1}^{n-1}\frac{a_m v_m}{z -
\gamma_m}\bigg|^2 + \frac{|a_{n}|^2 v_n^{2}}{|z - \gamma_n|^2}
 + \bigg|\sum_{m=n+1}^\infty \frac{a_m v_m}{z -
\gamma_m}\bigg|^2\biggr{]}d\mu(z) \nonumber\\
\lesssim \int_{\Omega_n} \left[|z|^{-2}\bigg(\sum_{m=1}^{n-1}|a_m
|v_m\bigg)^2+\bigg(\sum_{m=n+1}^\infty
\frac{|a_m|v_m}{|\gamma_m|}\bigg)^2\right] d\mu(z)+ |a_n|^2 v_n;
\end{eqnarray*}
here we used the Cauchy--Schwarz inequality and \eqref{trivial}.
Hence it remains for us to show that
\begin{equation}
\label{first1}\sum_{n=1}^\infty  \bigg(\sum_{m=1}^{n-1}|a_m
|v_m\bigg)^2 \int_{\Omega_n}|z|^{-2} d\mu(z) \lesssim
\sum_{j=1}^\infty |a_j|^2 v_j. \end{equation} and
\begin{equation}
\label{second2}\sum_{n=1}^\infty
\bigg(\sum_{m=n+1}^{\infty}\frac{|a_m |v_m}{|\gamma_m|}\bigg)^2
\mu(\Omega_n) \lesssim \sum_{j=1}^\infty |a_j|^2v_j. \end{equation}

We consider first \eqref{first1}. To simplify the writing, we set
\[ \tau_n=\left(\int_{\Omega_n} |z|^{-2} d\mu(z)\right)^{\frac{1}{2}}.\] By duality,
\[ \left(\sum_{n=1}^\infty \tau_n^2 \bigg(\sum_{m=1}^{n-1}|a_m
|v_m\bigg)^2\right)^{\half} =\sup_{\|(c_n)\|_{\ell^2}=1}
\sum_{n=1}^\infty |c_n| \tau_{n}\sum_{m=1}^{n-1}|a_m |v_m.\] Since
\[ \sum_{n=1}^\infty |c_n| \tau_{n}\sum_{m=1}^{n-1}|a_m |v_m
=\sum_{m=1}^\infty |a_m| v_m \sum_{n=m+1}^\infty |c_n| \tau_n,
\] it suffices to show that the $\ell^2$-norm of
\[\alpha_m =
v_m^{\half}\sum_{n=m+1}^{\infty}|c_n| \tau_{n}\] is bounded by a
constant times the $\ell^2$-norm of $(c_n)$. To this end, we note
that the Cauchy--Schwarz inequality gives
\[ |\alpha_m|^2 \le v_m \sum_{n=m+1}^\infty |c_n|^2 V_n^{-\half}
\sum_{j=m+1}^\infty \tau_j^2 V_j^{\half}.\]

By \eqref{carl_2}, we see that
\[
\sum_{j: 2^l V_m<V_j\leq 2^{l+1}V_m}\tau_j^2 V_j^{\half} \lesssim
\frac{1}{2^{\frac{l}{2}}V_{m+1}^{\half}}\] for $\l\ge 0$. Summing
these inequalities, we get
\[\sum_{j=m+1}^\infty\tau^2_jV^\half_j\lesssim
\frac{1}{V^\half_{m+1}}.\] Hence
\[ |\alpha_n|^2 \lesssim \frac{v_m}{V_{m+1}^{\half}} \sum_{n=m+1}^\infty |c_n|^2
V_n^{-\half}.\] This gives us \[ \sum_{n=1}^\infty |\alpha_n|^2
\lesssim \sum_{n=1}^\infty |c_n|^2
V_n^{-\half}\sum_{m=1}^{n-1}\frac{v_m}{V_{m+1}^{\half}}, \] and so
\eqref{first1} follows because
\[ V_n^{-\half}\sum_{m=1}^{n-1}\frac{v_m}{V_{m+1}^{\half}}
\le V_n^{-\half}\int_0^{V_n} x^{-\half} dx = 2. \]

We next consider \eqref{second2}. We note to begin with that the
Cauchy--Schwarz inequality gives
\[ \bigg(\sum_{m=n+1}^{\infty}
\frac{|a_m| v_m}{|\gamma_m|} \bigg)^2\le \sum_{m=n+1}^\infty |a_m|^2
v_m P_{m-1}^{\half} \sum_{j=n+1}^\infty \frac{v_j}{P_{j-1}^{\half}
|\gamma_j|^2}. \]
 Since
\[ \sum_{j=n+1}^\infty
\frac{v_j}{P_{j-1}^{\half} |\gamma_j|^2}\le \int_0^{P_n}x^{-\half}
dx \le 2 P_n^{\half},\] it follows that
\[ \sum_{n=1}^\infty \mu(\Omega_n) \bigg(\sum_{m=n+1}^\infty
\frac{|a_m|v_m^{\half}}{|\gamma_m|}\bigg)^2\lesssim
\sum_{n=1}^\infty \mu(\Omega_n) P_n^{\half} \sum_{m=n+1}^\infty
|a_m|^2v_m P_m^{\half},\] which becomes \[\sum_{n=1}^\infty
\mu(\Omega_n) \bigg(\sum_{m=n+1}^\infty
\frac{|a_m|v_m}{|\gamma_m|}\bigg)^2\lesssim \sum_{m=1}^\infty
|a_m|^2v_m P_{m-1}^{\half} \sum_{n=1}^{m-1} \mu(\Omega_n)
P_n^{\half}
\] when we change the order of summation. From \eqref{carl_2} it
follows that
\[\sum_{n=1}^{m-1}
\mu(\Omega_n) P_n^{\half}\lesssim \sum_{l=0}^{\infty}\sum_{n: 2^l
P_{m-1}\leq P_n\leq 2^{l+1}P_{m-1}}\mu(\Omega_n)
P_n^{\half}\lesssim\frac{1}{P^\half_{m-1}}
\sum_{l=0}^\infty\frac{1}{2^{\frac{l}{2}}}\lesssim
\frac{1}{P^\half_{m-1}},
\]
and we get \eqref{second2}.
\end{proof}

\subsection{Special cases} Condition  \eqref{trivial} of Theorem \ref{th1}
is a condition on the local behavior of $\mu$, while condition
\eqref{carl_2} deals with its global behavior. Combining the two
conditions, we see that \eqref{trivial} may be replaced by the
stronger global condition
\[ \sup_{n\ge 1} \int_{\CC} \frac{v_n d\mu(z)}{|z-\gamma_n|^2}
<\infty.\] We single out two cases in which \eqref{carl_2} is
automatically fulfilled once either this condition or the original
one \eqref{trivial} holds.
\begin{corollary}\label{corone}
Suppose the sequence $\Gamma$ satisfies the sparseness condition
\eqref{expon} and that the numbers $v_n$ grow at least exponentially
and that the numbers $v_n/|\gamma_n|^2$ decay at least exponentially
with $n$. If $\mu$ is a nonnegative measure on $\CC$ with
$\mu(\Gamma)=0$, then the operator $H_{(\Gamma,v)}$ is bounded from
$\ell^2_v$ to $L^2(\mathbb{C},\mu)$ if and only if
\[ \sup_{n\ge 1} \int_{\Omega_n} \frac{v_n d\mu(z)}{|z-\gamma_n|^2}
<\infty.\]
\end{corollary}
\begin{corollary}\label{cortwo}
Suppose the sequence $\Gamma$ satisfies the sparseness condition
\eqref{expon} and that $\sum_n v_n<\infty$. If $\mu$ is a
nonnegative measure on $\CC$ with $\mu(\Gamma)=0$, then the operator
$H_{(\Gamma,v)}$ is bounded from $\ell^2_v$ to $L^2(\mathbb{C},\mu)$
if and only if \[ \sup_{n\ge 1} \int_{\CC} \frac{v_n
d\mu(z)}{|z-\gamma_n|^2} <\infty.\]
\end{corollary}
Both corollaries follow immediately from Theorem \ref{th1}.

\subsection{Bessel sequences}\label{bessel}

We now switch to discrete Hilbert transforms of the form
$H_{(\Gamma,v);(\Lambda;w)}$ with $w$ the Bessel weight sequence for
$\Lambda$ with respect to $(\Gamma,v)$. As explained in
Subsection~\ref{systems}, this means that we will be dealing with
Bessel sequences for $\Hg(\Gamma,v)$.

We want to disentangle condition \eqref{carl_2}. To this end, we
split any given sequence $\Lambda$ into three disjoint sequences:
\[ \Lambda^{(0)}=\left\{\lambda\in \Lambda: \
\text{if $\lambda$ is in $\Omega_n$, then} \
\frac{v_n}{|\lambda-\gamma_n|^2}\ge
\max\left(\frac{V_n}{|\lambda|^2}, P_n\right)\right\}.\]
\[ \Lambda^{(V)}=\left\{\lambda\in \Lambda: \
\text{if $\lambda$ is in $\Omega_n$, then} \
\frac{V_n}{|\lambda|^2}> \max\left(\frac{v_n}{|\lambda-\gamma_n|^2},
P_n\right)\right\}.\]
\[ \Lambda^{(P)}=\left\{\lambda\in \Lambda: \
\text{if $\lambda$ is in $\Omega_n$, then} \ P_n >
\max\left(\frac{v_n}{|\lambda-\gamma_n|^2},\frac{V_n}{|\lambda|^2}\right)\right\}.\]
We say that a sequence $\Lambda$ is \emph{$V$-lacunary} if \[
\sup_{n} \# \left(\Lambda \cap \bigcup_{m:\ 2^n\le V_m\le 2^{n+1}}
\Omega_m\right)< \infty\] and \emph{$P$-lacunary} if
\[ \sup_{n} \# \left(\Lambda
\cap \bigcup_{m:\ 2^{-n-1}\le P_m\le 2^{-n}} \Omega_m\right)<
\infty.\]

\begin{thm}\label{thm3}
Suppose the sequence $\Gamma$ satisfies the sparseness condition
\eqref{expon} and that $v$ is an admissible weight sequence for
$\Gamma$. Let $\Lambda$ be a sequence in $(\Gamma,v)^*$, and let $w$
be the Bessel weight sequence for $\Lambda$ with respect to
$(\Gamma,v)$. Then $H_{(\Gamma,v);(\Lambda,w)}$ is a bounded
transformation if and only if $\sup_n \# (\Lambda\cap
\Omega_n)<\infty$, $\Lambda^{(V)}$ is a $V$-lacunary sequence,
$\Lambda^{(P)}$ is a $P$-lacunary sequence, and
\begin{equation}
 \sup_{n\ge 1} \left(V_n \sum_{m\ge n} \sum_{\lambda\in \Lambda^{(0)}\cap\Omega_m}
\frac{|\lambda-\gamma_m|^2}{v_m|\lambda|^{2}} +
  P_n \sum_{m\le n}\sum_{\lambda\in \Lambda^{(0)}\cap\Omega_m}
 \frac{|\lambda-\gamma_m|^2}{v_m}\right) <\infty.
 \label{carl_3}
\end{equation}
\end{thm}

This splitting into a ``super-thin" sequence $\Lambda^{(V)}\bigcup
\Lambda^{(P)}$ and  a ``distorted" sequence  $\Lambda^{(0)}$
represents a phenomenon not previously recorded, as far as we know.
  Corollary~\ref{corone} and Corollary~\ref{cortwo}, when
  restricted to the case of Bessel sequences, describe
two situations in which the ``super-thin'' part does not appear, for
different reasons: Corollary~\ref{corone} covers the case when $V_n$
grows exponentially and $P_n$ decays exponentially with $n$;
$\Lambda^{(V)}$ and $\Lambda^{(P)}$ can then both be ``absorbed'' in
$\Lambda^{(0)}$. Corollary~\ref{corone} covers the case when $V_n$
is uniformly bounded so that $\Lambda^{(V)}$ can only be a finite
sequence; the sequence $\Lambda^{(P)}$ can again be ``absorbed'' in
$\Lambda^{(0)}$.

We conclude that the most interesting situation occurs when either
$v_n/|\gamma_n|^2=o(P_n)$ or $v_n=o(V_n)$ and $V_n\to\infty$ as
$n\to\infty$. These two cases will be studied in depth in
Section~\ref{invertibility}.

\subsection{An example answering Baranov's
question}\label{cluster} We will now modify our construction to
obtain an example that gives a negative answer to Baranov's question
posed in Subsection~\ref{deBmod}.

We assume that $(t_n)$ is a sequence of positive numbers such that
\[ \inf_{n\ge 1} \frac{t_{n+1}}{t_n}>1.\]
In addition, we will assume that, for each positive integer $n$, we
have the following cluster of $n$ points:
\[\gamma_{n,l}=t_n+l-1, \ \ \ 1\leq l\leq n. \]
We denote this finite sequence by $\Gamma_n$ and set
\[ \Gamma=\bigcup_{n=1}^\infty \Gamma_n.\]
We will consider the simplest case when the corresponding weight
sequence $v$ is identically $1$, i.e., $v_{n,l}=1$ for every point
$\gamma_{n,l}$ in $\Gamma$.

It may be noted that if we want to describe the measures $\mu$ for
which $H_{(\Gamma,1)}$ is bounded from $\ell^2$ to $L^2(\CC\setminus
\Gamma,\mu)$, then it suffices to consider the behavior of $\mu$ in
the Carleson squares
\[S_n=\left\{z=x+iy: \ |x-\gamma_{n,1}|\le 2n,\ \ 0\le y\le 4 n\right\}. \]
Indeed, outside these squares, each cluster $\Lambda_n$ has
basically the same effect as if a single point were located at, say,
$\lambda_{n,1}$ with weight $n$. This means that Theorem~\ref{th1}
applies to describe the behavior of $\mu$ outside the squares $S_n$.
In fact, by this observation, one may obtain a complete solution to
the boundedness problem for these particular sequences $\Gamma$ and
$v$. We omit this description here and confine the discussion to a
suitable example solving Baranov's problem.

The preceding remarks indicate that the sequence $\Lambda$ should be
placed inside the union of the squares $S_n$. We set
\[\lambda_{n,s} = \gamma_{n,1}-2^s,\ \ \  0 \leq s\leq \log_2 n\] and
then $\Lambda_n=(\lambda_{n,s})_s$ with $s$ running from $0$ to
$[\log_2 n]$ (the integer part of $\log_2 n$), and
\[ \Lambda = \bigcup_{n=1}^\infty \Lambda_n.\]
We observe that we have
\[ w_{n,s}=\left(\sum_{m=1}^{\infty}\sum_{l=1}^n
\frac{1}{|\gamma_{m,l}-\lambda_{n,s}|^2}\right)^{-1}\simeq
|\lambda_{n,s}-\gamma_{n,1}|=2^s.\] These numbers constitute the
sequence $w$, which is the Bessel weight sequence for $\Lambda$ with
respect to $(\Gamma,v)$. We now make the following claim.

\begin{claim} If the sequences $\Gamma$, $\Lambda$, and $w$ are constructed as
above, then $H_{(\Gamma,1); (\Lambda,w)}$ is a bounded
transformation. \end{claim}

The interesting point, giving a negative answer to Baranov's
question, is that there are more than $\log_2 n$ points from
$\Lambda$ between the neighboring clusters $\Lambda_{n-1}$ and
$\Lambda_n$.

\begin{proof}[Proof of the claim]
Let $a=(a_{n,l})$ be an arbitrary $\ell^2$-sequence associated with
$\Gamma$ and set
\[
H_{(\Gamma,1)}a(\lambda)=\sum_{n=1}^\infty\sum_{l=1}^n
\frac{a_{n,l}}{\lambda-\gamma_{n,l}}.\] An application of the
Cauchy--Schwarz inequality gives
\[\sum_{s=0}^{[\log_2 n]} |H_{(\Gamma,1)}a(\lambda_{n,s})|^2 w_{n,s}
\lesssim \frac{n^3}{t_n^2}\|a\|_{\ell^2}^2 +\sum_{s=0}^{[\log_2 n]}
2^s \biggl{(}\sum_{l=1}^n
\frac{|a_{n,l}|}{|\lambda_{n,s}-\gamma_{n,l}|} \biggr{)}^2 .\] The
summation over $n$ of the first term on the right-hand side causes
no problem because $t_n$ grows at least exponentially with respect
to $n$. We therefore concentrate on the second term
\[ A_n=\sum_{s=0}^{[\log_2
n]}2^s \biggl(\sum_{l=1}^n\frac{|a_{n,l}|}{2^s+l-1} \biggr)^2.\] The
Cauchy--Schwarz inequality gives
\[ \biggl(\sum_{l=1}^n\frac{|a_{n,l}|}{2^s+l-1} \biggr)^2
\le \sum_{j=1}^n \frac{j^{-\half}}{2^s+j-1} \sum_{l=1}^n
\frac{l^{\half} |a_{n,l}|^2}{2^s+l-1}.\] The first of the two sums
on the right-hand side is bounded by a constant times
$2^{-\frac{s}2}$, and so it follows that
\[ A_n\le \sum_{s=0}^{[\log_2
n]}2^{\frac{s}2} \sum_{l=1}^n \frac{l^{\half}
|a_{n,l}|^2}{2^s+l-1}.\] Changing the order of summation and using
that
 \[\sum_{s=0}^{[\log_2
n]}\frac{2^{\frac{s}2}}{2^s+l-1}\lesssim l^{-1\slash2},\] we finally
obtain the desired estimate:
\[ A_n\lesssim \sum_{j=1}^n |a_{n,j}|^2. \]
\end{proof}

\section{Proof of Theorem~\ref{th12}}\label{feichtinger}

We use the same splitting as in Theorem~\ref{thm3} and treat the
three sequences $\Lambda^{(0)}$, $\Lambda^{(V)}$, and
$\Lambda^{(P)}$ separately. We use Lemma~\ref{duality}, i.e., we
make a splitting so that, for each subsequence $\Lambda'$ with
associated weight sequence, the adjoint transformation
$H_{(\Lambda',w');(\Gamma,v)}$ is bounded below. From now on, we
will use the notations \begin{equation}\label{WQ}
W_n=\sum_{m=1}^{n-1} w_m \ \ \ \text{and} \ \ \
Q_n=\sum_{m=n+1}^\infty \frac{w_m}{|\lambda_m|^2}.
\end{equation}

\subsection{The splitting of $\Lambda^{(0)}$} We may assume that there
is at most one point $\lambda_n$ in $\Lambda^{(0)}$ from each
annulus $\Omega_n$; we denote the corresponding weights by $w_n$.
Let $\Lambda'=(\lambda_{n_j})$ be a subsequence of $\Lambda^{(0)}$
with corresponding weight sequence $w'=(w_{n_j})$, and let
$a=(a_{n_j})$ be an arbitrary $\ell^2_{w'}$-sequence. Since
\[ |\xi-\eta|^2\ge |\xi|^2-2|\xi| |\eta| + |\eta|^2\ge \half
|\xi|^2-|\eta|^2\] for arbitrary complex numbers $\xi$ and $\eta$,
we have \[ |H_{(\Lambda',w');(\Gamma,v)}a(\gamma_{n_j})|^2 \geq
\half \frac{|a_{n_j}|^2 w_{n_j}^{2} }{|\lambda_{n_j} -
\gamma_{n_j}|^2}-2 \bigg|\sum_{l=1}^{j-1}\frac{a_{n_l}
w_{n_l}}{\overline{\gamma_{n_j}} - \overline{\lambda_{n_l}}}\bigg|^2
-
 2 \bigg|\sum_{l=j+1}^\infty \frac{a_{n_l} w_{n_l}}
 {\overline{\gamma_{n_j}} - \overline{\lambda_{n_l}}}\bigg|^2.\]
  On the other hand,
$w_{n_j}\simeq |\lambda_{n_j}-\gamma_{n_j}|^2/v_{n_j}$. Therefore,
by the definition of $\Lambda^{(0)}$, there is a positive constant
$c$ such that
\begin{equation} \|H_{(\Lambda',w');(\Gamma,v)}a\|_v^2\ge c\|a\|_{w'}^2- 2\
\sum_{j=1}^\infty \bigg[\bigg|\sum_{l=1}^{j-1}\frac{a_{n_l}
w_{n_l}}{\overline{\gamma_{n_j}} - \overline{\lambda_{n_l}}}\bigg|^2
+
  \bigg|\sum_{l=j+1}^\infty \frac{a_{n_l} w_{n_l}}
  {\overline{\gamma_{n_j}} - \overline{\lambda_{n_l}}}\bigg|^2\bigg] v_{n_j}.
\label{identical}\end{equation} Hence it remains for us to show
that, for a given $\varepsilon>0$, we may obtain
\begin{equation}
\label{first11}\sum_{j=1}^\infty  \bigg(\sum_{l=1}^{j-1}|a_{n_l}
|w_{n_l}\bigg)^2 \frac{v_{n_j}}{|\lambda_{n_j}|^2} \le \varepsilon
\sum_{j=1}^\infty |a_{n_j}|^2 w_{n_j}
\end{equation} and
\begin{equation}
\label{second22}\sum_{j=1}^\infty
\bigg(\sum_{l=j+1}^{\infty}\frac{|a_{n_l}
|w_{n_l}}{|\lambda_{n_l}|}\bigg)^2 v_{n_j} \le \varepsilon
\sum_{j=1}^\infty |a_{n_j}|^2 w_{n_j} \end{equation} for every
subsequence $\Lambda'$ in a finite splitting of $\Lambda^{(0)}$.

We proceed as in the proof of Theorem~\ref{th1}. Thus we set
$\tau_j=v_{n_j}^{\frac{1}{2}}/|\lambda_{n_j}|$ and consider first
\eqref{first11}. By duality,
\[ \left(\sum_{j=1}^\infty \tau_j^2 \bigg(\sum_{l=1}^{j-1}|a_{n_l}|
w_{n_l}\bigg)^2\right)^{\half} =\sup_{\|(c_j)\|_{\ell^2}=1}
\sum_{j=1}^\infty |c_j| \tau_{j}\sum_{l=1}^{j-1}|a_{n_l} |w_{n_l}.\]
Since
\[ \sum_{j=1}^\infty |c_j| \tau_{j}\sum_{l=1}^{j-1}|a_{n_l} |w_{n_l}
=\sum_{l=1}^\infty |a_{n_l}| w_{n_l} \sum_{j=l+1}^\infty |c_j|
\tau_j,
\] it suffices to show that the $\ell^2$-norm of
\[\alpha_l =
w_{n_l}^{\half}\sum_{j=l+1}^{\infty}|c_j| \tau_{j}\] can be made
smaller than $\varepsilon$ times the $\ell^2$-norm of $(c_j)$. To
this end, we note that the Cauchy--Schwarz inequality gives
\[ |\alpha_l|^2 \le w_{n_l} \sum_{j=l+1}^\infty |c_j|^2 W_{n_j}^{-\half}
\sum_{m=l+1}^\infty \tau_m^2 W_{n_m}^{\half}.\] Using
\eqref{carl_2}, we get
\[\sum_{m=l+1}^\infty\tau^2_m W^\half_{n_m}\lesssim
\frac{1}{W^\half_{n_{l+1}}}.\] Hence
\[ |\alpha_l|^2 \lesssim \frac{w_{n_l}}{W_{n_{l+1}}^{\half}}
\sum_{j=l+1}^\infty |c_j|^2 W_{n_j}^{-\half} .\] This gives us \[
\sum_{l=1}^\infty |\alpha_l|^2 \lesssim \sum_{j=1}^\infty |c_j|^2
W_{n_j}^{-\half}\sum_{l=1}^{j-1}\frac{w_{n_l}}{W_{n_{l+1}}^{\half}},
\] and so \eqref{first11} would follow if we could obtain
\begin{equation}
\sum_{l=1}^{j-1}\frac{w_{n_l}}{W_{n_{l+1}}^{\half}} \le c
\varepsilon W_{n_j}^{\half} \label{small1} \end{equation} for an
absolute constant $c$.

Having singled out this goal, we proceed to consider
\eqref{second22}. We note to begin with that the Cauchy--Schwarz
inequality gives
\[ \bigg(\sum_{l=j+1}^{\infty}
\frac{|a_{n_l}| w_{n_l}}{|\lambda_{n_j}|} \bigg)^2\le
\sum_{l=j+1}^\infty |a_{n_l}|^2 w_{n_l} Q_{n_{l-1}}^{\half}
\sum_{m=j+1}^\infty \frac{w_{n_m}}{Q_{n_{m-1}}^{\half}
|\lambda_{n_m}|^2}.
\]
Now our goal will be to obtain \begin{equation} \sum_{m=j+1}^\infty
\frac{w_{n_m}}{Q_{n_{m-1}}^{\half} |\lambda_{n_m}|^2}\le c
\varepsilon Q_{n_j}^{\half}. \label{small2}\end{equation} Indeed,
this would imply
\[ \begin{split} \sum_{j=1}^\infty \bigg(\sum_{l=j+1}^\infty
\frac{|a_{n_l}|w_{n_l}}{|\lambda_{n_l}|}\bigg)^2 v_{n_j} & \lesssim
\varepsilon \sum_{j=1}^\infty v_{n_j} Q_{n_j}^{\half}
\sum_{l=j+1}^\infty |a_{n_l}|^2w_{n_l} Q_{n_{l-1}}^{\half}\\
& =  \varepsilon \sum_{l=1}^\infty |a_{n_l}|^2w_{n_l}
Q_{{n_{l-1}}}^{\half} \sum_{j=1}^{l-1} v_{n_j} Q_{n_j}^{\half}.
\end{split}
\] By \eqref{carl_2}, we have
\[\sum_{j=1}^{l-1}
v_{n_j} Q_{n_j}^{\half}\lesssim \frac{1}{Q^\half_{n_{l-1}}},
\]
and so it will suffice to have \eqref{small2}.

In order to obtain the two estimates \eqref{small1} and
\eqref{small2} for every subsequence in our finite splitting of
$\Lambda^{(0)}$, we make a splitting according to the following
algorithm:
\begin{itemize}
\item[(1)] Let $\delta$ be a small positive number to be chosen later.
Select those $n$ for which $w_n> \delta W_n$. If we choose
$\Lambda'$ to consist of every $N$-th $\lambda_n$ in the
corresponding subsequence of $\Lambda^{(0)}$, then we get
\[
\sum_{l=1}^{j-1}\frac{w_{n_l}}{W_{n_{l+1}}^{\half}} \le
\frac{2}{\delta N} W_{n_j}^{\half} \] by again comparing the sum to
the integral of the function $x^{-\half}$ over the interval from $0$
to $W_{n_j}$. Thus we achieve our goal if we choose $N$ to be of the
order of magnitude $1/(\delta \varepsilon)$.
\item[(2)] Return to those points $\lambda_{n_j}$ not selected in (1).
For these we have $w_{n_j}\le \delta W_{n_j}$. Group these points
into blocks of points with consecutive indices such that for each
block $\delta \le \sum w_{n_j} W^{-1}_{n_j} < 2\delta$. Construct
new subsequences by picking every $N$-th block from this sequence of
blocks. Then some elementary estimates, again using comparisons with
an integral, lead to the following inequality:
\[ \sum_{l=1}^{j-1}
\frac{w_{n_l}}{W_{n_{l+1}}^{\half}} \le
\frac{16\delta}{1-(1-2\delta)^N} W_{n_j}^{\half},
\]
where we sum over the new subsequence. Thus it would suffice if we
choose $N$ to be roughly $1/\delta$ and $\delta$ to be a suitable
constant times $\eps$.
\item[(3)] Take one of the subsequences selected in (1) or (2) and consider
the subsequence of this subsequence, say $\Lambda'=(\lambda_{n_j})$,
along which $ w_{n_j} |\lambda_{n_j}|^{-2} > \delta Q_{n_j}$. If we
select a new subsequence by picking every $N$-th $\lambda_{n_j}$ in
the sequence $\Lambda'$, then the sum in \eqref{small2} becomes
smaller than $2/(\delta N) Q_{n_j}$ by the same argument as in (1).
Again our goal is achieved if we choose $N$ to be of the order of
magnitude $1/(\delta \varepsilon)$.
\item[(4)] Take again one of the subsequences selected in (1) or (2) and consider
those subsequences of these for which we have $w_{n_j}
|\lambda_{n_j}|^{-2} \le \delta Q_{n_j}$. Group the points in these
subsequences into blocks of points with consecutive indices such
that for each block $\delta \le \sum w_{n_j}|\lambda_{n_j}|^{-2}
Q_{n_j}^{-1} < 2\delta$. Now construct new subsequences by picking
every $N$-th block from this sequence of blocks. Then as in point
(2) we get
\[ \sum_{m=j+1}^\infty
\frac{w_{n_m}}{Q_{n_{m-1}}^{\half} |\lambda_{n_m}|^2}\le
\frac{16\delta}{1-(1-2\delta)^N} Q_{n_j}^{\half}.
\]
(Here the summation is again over the new subsequence.) We observe
once more that it would suffice if we choose $N$ to be roughly
$1/\delta$ and $\delta$ to be a suitable constant times $\eps$.
\end{itemize}

\subsection{The splitting of $\Lambda^{(V)}$} The splitting of $\Lambda^{(V)}$ is almost
identical to that of $\Lambda^{(0)}$. We will now use the estimate
\begin{equation} |H_{(\Lambda',w');(\Gamma,v)}a(\gamma_{n})|^2 \geq \half
\frac{|a_{n_j}|^2 w_{n_j}^{2} }{|\overline{\lambda_{n_j}} -
\overline{\gamma_{n}}|^2}-2 \bigg|\sum_{l=1}^{j-1}\frac{a_{n_l}
w_{n_l}}{\overline{\lambda_{n_j}} - \overline{\gamma_{n}}}\bigg|^2 -
 2 \bigg|\sum_{l=j+1}^\infty \frac{a_{n_l} w_{n_l}}{\overline{\gamma_{n}} -
\overline{\lambda_{n_l}}}\bigg|^2.\label{fund}\end{equation} The
reason we write `$\gamma_n$' instead of `$\gamma_{n_j}$' is that we
need to sum over several annuli $\Omega_n$ in order to estimate the
norm of $\|a\|_w$. Indeed, we may assume that $\lambda_{n_j}$
belongs to a union of annuli $\Omega_n$, denoted by $\Delta_j$, such
that
\[ \sum_{\gamma_n\in \Delta_j}
\frac{v_n}{|\lambda_{n_j} - \gamma_{n}|^2} \ge
\frac{1}{10}\frac{V_{n_j}}{|\lambda_j|^2},
\] with the sets $\Delta_j$ being pairwise disjoint. Therefore, by
the definition of $\Lambda^{(V)}$, there is a constant $c$ such that
\[ \sum_{\gamma_n\in \Delta_j}
|a_{n_j}|^2 w_{n_j}^{2}\, \frac{v_n }{|\lambda_{n_j} -
\gamma_{n}|^2} \ge c |a_{n_j}|^2 w_{n_j}.\] Hence we obtain
\[ \|H_{(\Lambda',w');(\Gamma,v)}a\|_v^2\ge c\|a\|_{w'}^2-
2\ \sum_{j=1}^\infty \sum_{\gamma_n \in
\Delta_j}\bigg[\bigg|\sum_{l=1}^{j-1}\frac{a_{n_l}
w_{n_l}}{\overline{\lambda_{n_j}} - \overline{\gamma_{n}}}\bigg|^2 +
  \bigg|\sum_{l=j+1}^\infty \frac{a_{n_l} w_{n_l}}{\overline{\lambda_{n_j}} -
\overline{\gamma_{n}}}\bigg|^2\bigg] v_{n_j}. \] The splitting is
then done in essentially the same way as above, repeating the
reasoning based on the estimate \eqref{identical}.

\subsection{The splitting of $\Lambda^{(P)}$} We use once more
\eqref{fund}, but this time we may assume that $\lambda_{n_j}$
belongs to a union of annuli $\Omega_n$, again denoted by
$\Delta_j$, such that
\[ \sum_{\gamma_n\in \Delta_j}
\frac{v_n}{|\lambda_{n_j} - \gamma_{n}|^2} \ge \frac{1}{10} P_{n_j},
\] with the sets $\Delta_j$ being pairwise disjoint. Therefore, by
the definition of $\Lambda^{(P)}$, there is a constant $c$ such that
\[ \sum_{\gamma_n\in \Delta_j}
|a_{n_j}|^2 w_{n_j}^{2}\, \frac{v_n }{|\lambda_{n_j} -
\gamma_{n}|^2} \ge c |a_{n_j}|^2 w_{n_j}.\] Hence we obtain
\[ \|H_{(\Lambda',w');(\Gamma,v)}a\|_v^2\ge c\|a\|_{w'}^2-
2\ \sum_{j=1}^\infty \sum_{\gamma_n \in
\Delta_j}\bigg[\bigg|\sum_{l=1}^{j-1}\frac{a_{n_l}
w_{n_l}}{\overline{\lambda_{n_j}} - \overline{\gamma_{n}}}\bigg|^2 +
  \bigg|\sum_{l=j+1}^\infty \frac{a_{n_l} w_{n_l}}{\overline{\lambda_{n_j}} -
\overline{\gamma_{n}}}\bigg|^2\bigg] v_{n_j}, \] and proceed as
outlined in the previous paragraph.

\section{The invertibility problem}\label{invertibility}

\subsection{Proof of Theorem~\ref{th13}}

It is clear that if the mapping $H_{(\Gamma,v);(\Lambda,w)}$ is
invertible, then $\Lambda$ is an exact uniqueness sequence for
$H_{(\Gamma,v)}$, which in turn implies that there is a unique
element $e=(e_n)$ in $\ell^2_v$ such that $H_{(\Gamma,v)}e$ vanishes
on $\Gamma\setminus \{\lambda_1\}$ and takes the value $1$ at
$\lambda_1$. Then
\[
G(z)=(z-\lambda_1)\sum_{n=1}^\infty \frac{e_n v_n}{z-\gamma_n}
\] is a generating function for $\Lambda$. Since by assumption
$G(\lambda_j)=0$ for $j>1$, we may write
\[ G(z)=G(z)-G(\lambda_j)=(z-\lambda_j)\sum_{n=1}^\infty
\frac{e_n v_n(\gamma_n-\lambda_1)}{(\gamma_n-\lambda_j)
(z-\gamma_n)},\] where on the right-hand side we have just
subtracted the respective series that define $G(z)$ and
$G(\lambda_j)$. Since $G$ is a generating function for $\Lambda$, it
follows that
\[ \sum_{n=1}^\infty \frac{|e_n|^2
|\gamma_n-\lambda_1|^2v_n}{|\gamma_n-\lambda_j|^2} <\infty. \] In
particular, the sequence
\[ e^{(j)}=\left( e_n \frac{\gamma_n-\lambda_1}{\gamma_n-\lambda_j}
\left(\sum_{m=1}^\infty \frac{e_m
v_m(\lambda_1-\gamma_m)}{(\lambda_j-\gamma_m)^2}\right)^{-1}\right)_n
\] will be the unique vector in $\ell^2_v$ such that
$H_{(\Gamma,v)}e^{(j)}(\lambda_l)$ is $0$ when $l\neq j$ and $1$ for
$l=j$.

To simplify the writing, we set
\[ \alpha_j =\left(\sum_{m=1}^\infty \frac{e_m
v_m(\lambda_1-\gamma_m)}{(\lambda_j-\gamma_m)^2}\right)^{-1}; \]
thus if $b=(b_1,b_2,..., b_l,0,0, ...)$ is a sequence with only
finitely many nonzero entries, then the sequence
\begin{equation}\label{linear} a=\left( e_n
(\gamma_n-\lambda_1)\sum_{j=1}^l\frac{b_j
\alpha_j}{\gamma_n-\lambda_j}\right)_n
\end{equation}
will be the unique vector in $\ell^2_v$ such that
$H_{(\Gamma,v);(\Lambda,w)}a=b$. This means that we have identified
a linear transformation, defined on a dense subset of $\ell^2_w$,
that must be the inverse transformation to
$H_{(\Gamma,v);(\Lambda,w)}$, should it exist. Hence, under the
assumption that $\Lambda$ is an exact uniqueness sequence for
$H_{(\Gamma,v)}$, a necessary and sufficient condition for
invertibility of $H_{(\Gamma,v);(\Lambda,w)}$ is that the linear
transformation defined by \eqref{linear} extends to a bounded
transformation on $\ell^2_w$. An equivalent condition is that the
transformation $H_{(\Lambda,\varpi);(\Gamma,\nu)}$ be bounded, where
\[ \nu_n=v_n |\lambda_1-\gamma_n|^2|e_n|^2  \] and
\[
\varpi_j=w_j^{-1} \left| \sum_{n=1}^\infty \frac{e_n
v_n(\lambda_1-\gamma_n)}{(\lambda_j-\gamma_n)^2}\right|^{-2}=
w_j^{-1} |\lambda_j-\lambda_1|^{-2}\left| \sum_{n=1}^\infty
\frac{e_n v_n}{(\lambda_j-\gamma_n)^2}\right|^{-2}.\]In the final
step, we used the definition of the sequence $(e_n)$.

\subsection{Localization of $\Lambda$ when $\Gamma$ is a sparse sequence}
\label{localsub}

We will for the rest of this section consider two interesting
special cases. The main point of this subsection will be that,
although $\Lambda$ may possibly have a nontrivial splitting into
three sequences $\Lambda^{(0)}$, $\Lambda^{(V)}$, $\Lambda^{(P)}$
(cf. the discussion in Subsection~\ref{bessel}), the invertibility
of $H_{(\Gamma,v);(\Lambda,w)}$ forces the sequences $\Lambda^{(V)}$
and $\Lambda^{(P)}$ to be trivial, in a sense to be made precise.

We assume as before that $\Gamma=(\gamma_n)$ is indexed by the
positive integers, and that the sequence is sparse in the sense that
\eqref{expon} holds. We retain the notation
\[ V_n=\sum_{m=1}^{n-1} v_m \ \ \ \text{and} \ \ \
P_n=\sum_{m=n+1}^{\infty} \frac{v_m}{|\gamma_m|^2} \] from the
previous section. In the discussion below, the sets
\[ D_n(v; M)=\left\{\lambda\in\Omega_n:
\ \frac{M v_n}{|\lambda-\gamma_n|^2}\ge
\max\left(\frac{V_n}{|\lambda|^2}, P_n\right)\right\},\] defined for
every admissible weight sequence $v$ and positive number $M$, will
play an essential role. If $M$ is fixed and either $v_n=o(V_n)$ or
$v_n/|\gamma_n|^2=o(P_n)$ when $n\to\infty$, then these sets are
essentially disks centered at $\gamma_n$ with radii that are
$o(|\gamma_n|)$ when $n\to \infty$. In such situations, the
splitting of a sequence $\Lambda$ into the three sequences
$\Lambda^{(0)}$, $\Lambda^{(V)}$, $\Lambda^{(P)}$ may be nontrivial,
in the sense that $\Lambda\setminus \bigcup_n D_n(v;M)$ may be an
infinite sequence for every positive $M$.

We will assume that $\Lambda=(\lambda_n)$ is a sequence disjoint
from $\Gamma$, indexed by a sequence of integers $(n_0, n_0+1,
n_0+2, ...)$ and ordered such that the moduli $|\lambda_n|$ increase
with $n$. For convenience, we assume that $\lambda_{n_0}\neq 0$. The
choice of $n_0$ is made such that $\Lambda$ is ``aligned'' with
$\Gamma$. More precisely, we will say that $\Lambda$ is a
\emph{$v$-perturbation of} $\Gamma$ if $n_0$ can be chosen such
that, for a sufficiently large $M$, $\lambda_n$ is in $D_n(v;M)$ for
all but possibly a finite number of indices $n$. If $\Lambda$ is a
$v$-perturbation of $\Gamma$, it will be implicitly understood that
$n_0$ is chosen so that the two sequences are ``aligned'' in this
way.

A $v$-perturbation $\Lambda$ of $\Gamma$ will be said to be,
respectively
\begin{itemize}
\item[-] an \emph{exact $v$-perturbation of} $\Gamma$ if $n_0=1$;
\item[-] a \emph{$v$-perturbation of $\Gamma$ of deficiency} $n_0-1$ if $n_0>1$;
\item[-] a \emph{$v$-perturbation of $\Gamma$ of excess} $1-n_0$ if $n_0<1$.
\end{itemize}

The main results of this subsection are the following two lemmas.

\begin{lemma}\label{localize}
Suppose $w$ is the Bessel weight sequence for $\Lambda$ with respect
to $(\Gamma,v)$ and that $v_n=o(V_n)$ when $n\to\infty$. If, in
addition, the transformation $H_{(\Gamma,v);(\Lambda,w)}$ is
invertible, then $\Lambda$ is either an exact $v$-perturbation of
$\Gamma$ or a $v$-perturbation of deficiency $1$.
\end{lemma}

\begin{lemma}\label{localize2}
Suppose $w$ is the Bessel weight sequence for $\Lambda$ with respect
to $(\Gamma,v)$ and that $v_n/|\gamma_n|^2=o(P_n)$ when
$n\to\infty$. If, in addition, the transformation
$H_{(\Gamma,v);(\Lambda,w)}$ is invertible, then $\Lambda$ is either
an exact $v$-perturbation of $\Gamma$ or a $v$-perturbation of
$\Gamma$ of excess $1$.
\end{lemma}

Note the contrast between these results and Theorem~\ref{thm3};
$\Lambda$ has no nontrivial $V$-lacunary or $P$-lacunary
subsequences when $H_{(\Gamma,v);(\Lambda,w)}$ is an invertible
transformation. We will see in the next section that, quite
remarkably, all three cases---exactness, deficiency $1$, and excess
$1$---may occur.

The proof of the two lemmas require several steps. We begin with a
simple estimate, to be used repeatedly in what follows. It concerns
the quantity
\[ \varrho_n =\prod_{m=\max(1,
n_0)}^{n} \frac{|\gamma_m|^2}{|\lambda_m|^2}, \] which will appear
prominently in our conditions for invertibility. We use again the
notation introduced in \eqref{WQ}, i.e., we set
\[
W_n=\sum_{m=n_0}^{n-1} w_m \ \ \ \text{and} \ \ \
Q_n=\sum_{m=n+1}^{\infty} \frac{w_m}{|\lambda_m|^2}.
\]
\begin{lemma}\label{varrhoest} If $\Lambda$ is a $v$-perturbation of $\Gamma$ and
$|\gamma_n|\simeq |\lambda_n|$, then we have both
\begin{equation}\label{lembasic1}
\left|\log\frac{\varrho_m}{\varrho_n}\right|^2\lesssim
(V_{m+1}-V_{n+1})(Q_n-Q_m) \end{equation} and
\begin{equation}\label{lembasic2}
\left|\log\frac{\varrho_m}{\varrho_n}\right|^2\lesssim
(W_{m+1}-W_{n+1})(P_n-P_m) \end{equation} when $m>n$. If, in
addition, either $v_n=o(V_n)$ or $v_n/|\gamma_n|^2=o(P_n)$ when
$n\to \infty$, then $\log \varrho_n=o(n)$ when $n\to \infty$.
\end{lemma}
\begin{proof}Since $|\gamma_n|\simeq |\lambda_n|$, we have
\begin{equation}\label{taylor}
\biggl{|}\log\frac{\varrho_m}{\varrho_n}\biggr{|}=2
\biggl{|}\sum_{l=n+1}^m\log\frac{|\gamma_l|}{|\lambda_l|}\biggr{|}\lesssim
\sum_{l=n+1}^m\biggl{|}1-\frac{|\gamma_l|}{|\lambda_l|}\biggr{|}.\end{equation}
Hence, by the Cauchy--Schwarz inequality, we get
\[ \biggl{|}\log\frac{\varrho_m}{\varrho_n}\biggr{|}^2\lesssim
\sum_{l=n+1}^m v_l \sum_{j=n+1}^m\frac{|\gamma_j-\lambda_j|^2}{v_j
|\lambda_j|^2}, \] which is the desired estimate \eqref{lembasic1}.
Another application of the Cauchy--Schwarz inequality to
\eqref{taylor} gives  \[
\biggl{|}\log\frac{\varrho_m}{\varrho_n}\biggr{|}^2\lesssim
\sum_{l=n+1}^m \frac{|\gamma_l-\lambda_l|^2}{v_l} \sum_{j=n+1}^m
\frac{v_j}{|\gamma_j|^2},
\]
which is the second estimate \eqref{lembasic2}.

Finally, starting again from \eqref{taylor} and using the
Cauchy--Schwarz inequality a third time, we get
\[ |\log \varrho_n|^2 \lesssim n \, \sum_{l=\max(1,n_0)}^n
\frac{|\gamma_l-\lambda_l|^2}{|\lambda_l|^2} \lesssim n\
\sum_{l=\max (1,n_0)}^n
\min\left(\frac{v_l}{V_l},\frac{v_l}{|\gamma_l|^2 P_l}\right),
\]
where in the last step we used that $\Lambda$ is a $v$-perturbation
of $\Gamma$. This relation gives the last statement in the lemma,
namely  that $\log \varrho_n =o(n)$ when  either $v_n=o(V_n)$ or
$v_n/|\gamma_n|^2=o(P_n)$ as $n\to \infty$.
\end{proof}

We next prove the following lemma, which is really a corollary to
Theorem~\ref{th1}.

\begin{lemma}
Suppose that either $v_n=o(V_n)$ or $v_n/|\gamma_n|^2=o(P_n)$ when
$n\to\infty$. If, in addition, $\mu$ is a nonnegative measure on
$\CC$ with $\mu(\Gamma)=0$ and the map $H_{(\Gamma,v)}$ is both
bounded and bounded below from $\ell^2_v$ to $L^2(\CC, \mu)$, then
there exist positive numbers $M$ and $\delta$ such that
\[\int_{D_n(v;M)} \frac{v_n d\mu(z)}{|z-\gamma_n|^2} \ge \delta\]
for all but finitely many indices $n$.
\end{lemma}

\begin{proof}
Applying the assumption about boundedness below to any sequence with
only one nonzero entry, we find that there is a positive number
$\sigma$ independent of $n$ such that
\[ \int_\mathbb{C}\frac{v_n d\mu(z)}
{|z-\gamma_n|^2} \geq\sigma \] for every $n$.  On the other hand,
since $|\gamma_n|$ grows at least exponentially and $H_{(\Gamma,v)}$
is bounded from $\ell^2_v$ to $L^2(\CC,\mu)$, we have
\[\sum_{m=n+1}^{\infty}\int_{\Omega_m}\frac{v_n d\mu(z)}
{|z-\gamma_n|^2}\lesssim v_n\, \sum_{m=n+1}^{\infty}\int_{\Omega_m}
\frac{d\mu(z)}{|z|^2}\lesssim \min\left(\frac{v_n}{V_n},
\frac{v_n}{|\gamma_n|^2 P_n}\right)\] and
\[\sum_{m=1}^{n-1}\int_{\Omega_m}\frac{v_n d\mu(z)}{|z-\gamma_n|^2}
\lesssim \frac{v_n}{|\gamma_n|^2}\sum_{m=1}^{n-1} \int_{\Omega_m}
d\mu(z)\lesssim \min\left(\frac{v_n}{V_n}, \frac{v_n}{|\gamma_n|^2
P_n}\right).\] We also have
\[\int_{\Omega_n\setminus D_n(v;M)}
\frac{v_n d\mu(z)}{|z-\gamma_n|^2} \le \frac{1}{M} \int_{\Omega_n}
\max \left(\frac{V_n}{|\lambda|^2}, P_n \right) d\mu(z)\lesssim
\frac{1}{M}, \] again using the condition for boundedness of  the
map $H_{(\Gamma,v)}: \ell^2_v\to L^2(\CC, \mu)$. The result follows
with $\delta=\sigma/2$ if we choose a sufficiently large $M$.
\end{proof}

The preceding lemma shows that if the transformation
$H_{(\Gamma,v);(\Lambda,w)}$ is invertible, then $\Lambda$ must
contain a subsequence that is a $v$-perturbation of $\Gamma$. The
next two lemmas show that $\Lambda$ itself must be a
$v$-perturbation of $\Gamma$.
\begin{lemma}\label{unique1}
Suppose that $v_n=o(V_n)$ when $n\to\infty$. If, in addition,
$\Lambda$ is an exact $v$-perturbation of $\Gamma$, then $\Lambda$
is a uniqueness sequence for $H_{(\Gamma,v)}$.
\end{lemma}

\begin{proof}
We argue by contradiction. So suppose there is a nonzero vector
$a=(a_n)$ in $\ell^2_v$ such that $H_{(\Gamma,v)}a$ vanishes on
$\Lambda$. This means that there is a nonzero entire function $J(z)$
such that
\[ \sum_{n=1}^\infty \frac{a_n v_n} {z-\gamma_n}= J(z)
 \prod_{m=1}^\infty
\frac{1-z/\lambda_m}{1-z/\gamma_m} \] for every $z$ in $\CC\setminus
\Gamma$. If we now choose $M$ sufficiently large, then we have
\[ \frac{V_n}{|z|^2} \gtrsim |J(z)|^2 \varrho_n \]
for $z$ in $\Omega_n\setminus D_n(v;M)$. Since $v_n=o(V_n)$ when
$n\to\infty$, the left-hand side is bounded by $e^{-\delta n}$ for
some positive $\delta$, while, by Lemma~\ref{varrhoest},
$\varrho_n=e^{o(n)}$ when $n\to\infty$. Thus the maximum of $|J(z)|$
in $\Omega_n\setminus D_n(v;M)$ tends to $0$ when $n\to \infty$,
which is a contradiction unless $J(z)\equiv 0$.
\end{proof}

\begin{lemma}\label{unique2}
Suppose that $v_n/|\gamma_n|^2=o(P_n)$ when $n\to\infty$. If, in
addition, $\Lambda$ is a $v$-perturbation of $\Gamma$ of excess $1$,
then $\Lambda$ is a uniqueness sequence for $H_{(\Gamma,v)}$.
\end{lemma}

\begin{proof}
We argue again by contradiction and assume that there is a nonzero
vector $a=(a_n)$ in $\ell^2_v$ such that $H_{(\Gamma,v)}a$ vanishes
on $\Lambda$. In this case, it follows that there is a nonzero
entire function $J(z)$ such that
\[ \sum_{n=1}^\infty \frac{a_n v_n} {z-\gamma_n}= J(z)(z-\lambda_0)
 \prod_{m=1}^\infty
\frac{1-z/\lambda_m}{1-z/\gamma_m} \] for every $z$ in $\CC\setminus
\Gamma$. If we now choose $M$ sufficiently large, then we have
\[ P_n \gtrsim |J(z)|^2 |z|^2 \varrho_n \]
for $z$ in $\Omega_n\setminus D_n(v;M)$. Since
$v_n/|\gamma_n|^2=o(P_n)$ when $n\to\infty$, we have that
$P_n/|z|^2$ is bounded by $e^{-\delta n}$ for some positive number
$\delta$, while, by Lemma~\ref{varrhoest}, $\varrho_n=e^{o(n)}$ when
$n\to\infty$. Thus the maximum of $|J(z)|$ in $\Omega_n\setminus
D_n(v;M)$ tends to $0$ when $n\to \infty$, which is a contradiction
unless $J(z)\equiv 0$.
\end{proof}
We finally prove two lemmas that, together with the previous three
lemmas, give the precise restrictions stated in Lemma~\ref{localize}
and Lemma~\ref{localize2}.
\begin{lemma}\label{unique3}
Suppose that $v_n=o(V_n)$ when $n\to\infty$. If, in addition,
$\Lambda$ is a $v$-perturbation of $\Gamma$ of deficiency $2$, then
$\Lambda$ is not a uniqueness sequence for $H_{(\Gamma,v)}$.
\end{lemma}
\begin{proof}
We may write
\[
\frac{c}{(z-\gamma_1)(z-\gamma_2)} \prod_{m=3}^\infty
\frac{1-z/\lambda_m}{1-z/\gamma_n}=\sum_{n=1}^\infty \frac{a_n
v_n}{z-\gamma_n}+ h(z),
\]
where $h$ is an entire function and \[ |a_n|^2 v_n^2 \simeq
\frac{|\gamma_n-\lambda_n|^2}{|\gamma_n|^4} \varrho_n. \] Since
$\Lambda$ is a $v$-perturbation, we therefore get
\[ \sum_{n=1}^\infty |a_n|^2 v_n \lesssim \sum_{n=1}^\infty
\frac{\varrho_{n}}{|\gamma_n|^2 V_{n}} <\infty, \] where in the
final step we used that the ratio $\varrho_n/V_n$ grows at most
sub-exponentially. We then get
\[ |h(z)|^2 \lesssim \frac{\varrho_n}{|z|^4} + \frac{V_n}{|z|^2}
\]
when $z$ is in $D_n(v; M)$ with $M$ sufficiently large. Using again
that both $\varrho_n$ and $V_n$ grow at most sub-exponentially, we
have that $h(z)\to 0$ when $z\to \infty$, which means that $h\equiv
0$.
\end{proof}

\begin{lemma}\label{unique4}
Suppose that $v_n/|\gamma_n|^2=o(P_n)$ when $n\to\infty$. If, in
addition, $\Lambda$ is a $v$-perturbation of $\Gamma$ of deficiency
$1$, then $\Lambda$ is not a uniqueness sequence for
$H_{(\Gamma,v)}$.
\end{lemma}

\begin{proof}
In this case, we may write
\[
\frac{c}{z-\gamma_1} \prod_{m=2}^\infty
\frac{1-z/\lambda_m}{1-z/\gamma_m}=\sum_{n=1}^\infty \frac{a_n
v_n}{z-\gamma_n}+ h(z),
\]
where $h$ is an entire function and \[ |a_n|^2 v_n^2 \simeq
\frac{|\gamma_n-\lambda_n|^2}{|\gamma_n|^2} \varrho_n. \] Since
$\Lambda$ is a $v$-perturbation, we get
\[ \sum_{n=1}^\infty |a_n|^2 v_n \lesssim \sum_{n=1}^\infty
\frac{\varrho_{n}}{|\gamma_n|^2 P_{n}} <\infty, \] where we now used
that the ratio $\varrho_n/P_n$ grows at most sub-exponentially. It
follows that
\[ |h(z)|^2 \lesssim \frac{\varrho_n}{|z|^2} + P_n
\]
when $z$ is in $D_n(v; M)$ with $M$ sufficiently large. We conclude
that $h(z)\to 0$ when $z\to \infty$, which means that $h\equiv 0$.
\end{proof}
\subsection{Geometric criteria for invertibility of $H_{(\Gamma,v);(\Lambda,w)}$ when
$\Gamma$ is a sparse sequence}\label{geometry}

After the preliminary results of the previous subsection, we may now
state our geometric conditions for invertibility. We begin with the
case when $v_n=o(V_n)$ as $n\to\infty$.

\begin{theorem}
Suppose $w$ is the Bessel weight sequence for $\Lambda$ with respect
to $(\Gamma,v)$ and that $V_n\to \infty$ and $v_n=o(V_n)$ when
$n\to\infty$. Then the transformation $H_{(\Gamma,v);(\Lambda,w)}$
is invertible if and only if $\sup_n V_n Q_n <\infty$ and one of the
following conditions holds:
\begin{itemize}
\item[(0)] $\Lambda$ is an exact $v$-transformation of $\Gamma$
and there are positive constants $C$ and $\delta$ such that
\begin{equation} \label{boundrho} \frac{\varrho_m}{\varrho_n}\le C
\left(\frac{V_m}{V_n}\right)^{1-\delta}\end{equation} whenever
$m>n$.
\item [(1)] $\Lambda$ is a $v$-transformation of $\Gamma$ of defect $1$
and there are positive constants $C$ and $\delta$ such that
\begin{equation} \label{boundrho2} \frac{\varrho_m}{\varrho_n}\ge C
\left(\frac{V_m}{V_n}\right)^{1+\delta}
\end{equation} whenever
$m>n$. \end{itemize} \label{y_int_th}
\end{theorem}
It is quite remarkable that the essential quantitative conditions
for invertibility, found in (0) and (1),  only depend on the moduli
of the complex numbers $\gamma_n/\lambda_n$.

We note that in the case when
\[ \sum_{n=1}^\infty v_n <\infty, \]
the result is much simpler and less delicate. Then, as can be seen
from the proof of part (0) of Theorem~\ref{y_int_th}, the
transformation $H_{(\Gamma,v);(\Lambda;w)}$ is invertible if and
only if $\Lambda$ is an exact $v$-transformation of $\Gamma$ and
$\sup_n Q_n < \infty$. This result can be viewed as a special case
of part (0) of the theorem.

To arrive at the results stated in Example 1 (see the introduction),
we note that if
\[ \frac{|\gamma_n-\lambda_n|}{|\lambda_n|} \lesssim \frac{v_n}{V_n}, \]
then
\[ Q_n=\sum_{m=n+1}^\infty \frac{w_m}{|\lambda_m|^2}\lesssim
\sum_{m=n+1}^\infty \frac{v_m}{V_{m+1}^2}\le \frac{1}{V_{n+1}}, \]
where in the last step we compared the sum with the integral of
$1/x^2$ from $V_{n+1}$ to $\infty$. We also have, assuming
$|\gamma_n|/|\lambda_n|-1\le c v_n/V_n$, that
\begin{equation}\label{integraltest} \log \frac{\varrho_m}{\varrho_n}\le 2c\,
(1+o(1))\sum_{j=n+1}^m \frac{v_l}{V_{l}}=2c(1+o(1)) \log
\frac{V_m}{V_n} \end{equation} when $m>n$ and $n\to \infty$. In view
of Theorem~\ref{y_int_th}, this gives part (0) of the example; part
(1) follows by the same argument, with the inequality in
\eqref{integraltest} reversed.

In the case when $v_n/|\gamma_n|^2=o(P_n)$, we have the following
counterpart to Theorem~\ref{y_int_th}.
\begin{theorem}
Suppose $w$ is the Bessel weight sequence for $\Lambda$ with respect
to $(\Gamma,v)$ and that $v_n/|\gamma_n|^2=o(P_n)$ when
$n\to\infty$. Then the transformation $H_{(\Gamma,v);(\Lambda,w)}$
is invertible if and only if we have $\sup_n W_n P_n <\infty$ and
one of the following two conditions holds:
\begin{itemize}
\item [(0)] $\Lambda$ is an exact $v$-transformation of $\Gamma$
and there are positive constants $C$ and $\delta$ such that
\begin{equation} \label{case0} \frac{\varrho_m}{\varrho_n}\ge C
\left(\frac{P_m}{P_n}\right)^{1-\delta} \end{equation} whenever
$m>n$.
\item[(1)]
$\Lambda$ is a $v$-transformation of $\Gamma$ of excess $1$ and
there are positive constants $C$ and $\delta$ such that
\begin{equation}\label{case1} \frac{\varrho_m}{\varrho_n}\le C
\left(\frac{P_m}{P_n}\right)^{1+\delta} \end{equation} whenever
$m>n$.
\end{itemize}
\label{y_int_th2}
\end{theorem}
There is a slight lack of symmetry between the two theorems; while
it may happen that $\sup_n V_n <\infty$, we will always have that
$P_n\to 0$. Therefore, no precaution is needed concerning the decay
of $P_n$.

We have the following statement, in complete analogy with Example 1
and with the same proof:

\begin{example2}Suppose that $v_n/|\gamma_n|^2=o(P_n)$ when $n\to\infty$ and
that $\sup_n W_n P_n <\infty$.
\begin{itemize}
\item[(0)] If, in addition, $\Lambda$ is an exact $v$-perturbation of $\Gamma$
and there is a real constant $c<1/2$ such that
\[
\frac{|\lambda_n|}{|\gamma_n|} - 1 \le c \, \frac{v_n}{|\gamma_n|^2
P_n}
\]
for all sufficiently large $n$, then $H_{(\Gamma,v);(\Lambda,w)}$ is
an invertible transformation.
\item[(1)] If, on the other hand, $\Lambda$ is a $v$-perturbation of $\Gamma$ of excess
$1$ and there is a positive constant $c >1/2$ such that
\[
\frac{|\lambda_n|}{|\gamma_n|} - 1 \ge c \,
\frac{v_n}{|\gamma_n|^2P_n}
\]
for all sufficiently large $n$, then $H_{(\Gamma,v);(\Lambda,w)}$ is
an invertible transformation.
\end{itemize}
\end{example2}

The two final subsections of this paper will present the proof of
Theorem~\ref{y_int_th}; the proof of Theorem~\ref{y_int_th2} is
completely analogous and will therefore be omitted.

\subsection{Proof of the necessity of the conditions in Theorem~\ref{y_int_th}}

In addition to the results of Subsection~\ref{localsub}, we will
need the following simple facts.

\begin{lemma}\label{simplefacts}
Let $c=(c_n)$ be a sequence of positive numbers.
\begin{itemize}
\item [(i)] If there is a constant $C$ such that $\sum_{m=1}^{n-1} c_m \le C c_n$
for $n>1$, then there is a positive constant $\delta$ such that $
c_m/c_n \ge C 2^{\delta (m-n)}$ whenever $m>n$.
\item [(ii)] If there is a constant $C$ such that $\sum_{m=n+1}^{\infty} c_m
\leq C c_n$ for every positive integer $n$, then there is a positive
constant $\delta$ such that $c_m/c_n \le C 2^{-\delta (m-n)}$
whenever $m>n$.\end{itemize}
\end{lemma}

\begin{proof}
We consider (i). The assumption implies that
\[ Nc_{n-1}\le N \sum_{m=1}^{n-1} c_m \le C \sum_{m=n}^{n+N-1} c_m \le C^2 c_{n+N}. \]
which means that if we choose $N>2C^2$, then $c_{n+j(N+1)}\ge 2^{j}
c_n$. The result follows if we choose $\delta=1/(N+2)$. The proof of
(ii) can be performed in a similar way.
\end{proof}

We turn to the proof of the necessity of the conditions in
Theorem~\ref{y_int_th}. Thus we begin by assuming that
$H_{(\Gamma,v);(\Lambda,w)}$ is an invertible transformation. Since
this means that, in particular, $H_{(\Gamma,v);(\Lambda,w)}$ is a
bounded transformation, we must have $\sup_n V_n Q_n <\infty$. Also,
in view of Lemma~\ref{localize}, we already know that $\Lambda$ is
either an exact $v$-perturbation of $\Gamma$ or a $v$-perturbation
of $\Gamma$ of deficiency $1$. Thus it remains only to establish the
necessity of the conditions in parts (0) and (1), under the
respective assumptions of exactness and deficiency $1$.

We treat the two cases separately:

(0) \emph{$\Lambda$ is an exact $v$-perturbation of $\Gamma$.} Since
$v_n=o(V_n)$, the weight sequence $w=(w_n)$ defined by
\eqref{admissurj} satisfies
\begin{equation}\label{simplew} w_n\simeq \frac{|\gamma_n-\lambda_n|^2}{v_n}.
\end{equation} As a consequence, we now obtain simple estimates for
the weight sequences $\nu=(\nu_n)$ and $\varpi=(\varpi_j)$ appearing
in Theorem~\ref{th13}.

We begin by noting that if $\Lambda$ is a $v$-perturbation of
$\Gamma$ and an exact uniqueness sequence for $H_{(\Gamma,v)}$, then
there is a constant $c$ such that
\begin{equation}\label{partfrac} \sum_{n=1}^\infty \frac{e_n v_n}{z-\gamma_n}
=\frac{c}{z-\gamma_1} \prod_{m=2}^\infty
\frac{1-z/\lambda_m}{1-z/\gamma_n} \end{equation} for every $z$ in
$\CC\setminus \Gamma$, where again $e=(e_n)$ is the vector such that
$H_{(\Gamma,v);(\Lambda,w)}e=(1,0,0,...)$. Indeed, the expression on
the left-hand side can have zeros only at the points $\lambda_m$ for
$m>1$, since $\Lambda$ is assumed to be an exact uniqueness sequence
for $H_{(\Gamma,v)}$. From \eqref{partfrac} we obtain
\[ |e_n|^2v_n^2 \simeq
\frac{|\lambda_n-\gamma_n|^2}{|\lambda_n|^2}\varrho_n,
\] and, therefore, using
\eqref{nu} and \eqref{simplew}, we obtain
\begin{equation}\label{nuest}\nu_n\simeq w_n \varrho_n.
\end{equation}

On the other hand, differentiating \eqref{partfrac} at
$z=\lambda_n$, we get
\[
\left| \sum_{l=1}^\infty \frac{e_l
v_l}{(\lambda_n-\gamma_l)^2}\right| \simeq
\frac{|\gamma_n|}{|\lambda_n|^2 |\lambda_n-\gamma_n|}
\prod_{m=1}^{n-1} \frac{|\gamma_m|}{|\lambda_m|}.\] Thus using
\eqref{varpi} and again \eqref{simplew}, we obtain
\begin{equation}\label{varpiest} \varpi_n\simeq v_n \varrho_n^{-1}.\end{equation}

To simplify the writing,  we set
\[ V^{(\varrho,0)}_n=\sum_{m=1}^{n-1} v_m \varrho^{-1}_m \ \ \
\text{and} \ \ \ P^{(\varrho,0)}_n=\sum_{m=n+1}^{\infty} v_m
|\lambda_m|^{-2}\varrho_m{-1} \] as well as
\[ W^{(\varrho,0)}_n=\sum_{m=1}^{n-1} w_n \varrho_n \ \ \
\text{and} \ \ \ Q^{(\varrho,0)}_n=\sum_{m=n+1}^{\infty} w_n
\varrho_n|\lambda_n|^{-2}.
\]
By Theorem~\ref{th13} and Theorem~\ref{th1}, we must have $\sup_n
V_n^{(\varrho,0)} Q_n^{(\varrho,0)} <\infty$; we will now show that
the estimate in part (0) is a consequence of this condition.

We set $n_1=2$ and define $n_j$ inductively by requiring
$V_{n_{j+1}-1}/V_{n_j}<2 \le V_{n_{j+1}}/V_{n_j}$. By
\eqref{lembasic1} of Lemma~\ref{varrhoest} and the uniform
boundedness of $V_nQ_n$, it follows that there are constants $c$ and
$C$ such that $c<\varrho_{n}/\varrho_{m} \le C$ when $n$ and $m$
both lie in the interval $[n_j, n_{j+1}]$. Hence we have
\begin{equation}\label{Vrhoest}
V_{n_j}^{(\varrho,0)}\simeq \sum_{l=1}^{j}V_{n_l}
\varrho_{n_l}^{-1}.
\end{equation}
Now if \begin{equation} \label{onone} Q_{n_j}-Q_{n_{j+1}} \ge
\frac{\varepsilon}{V_{n_{j+1}}},
\end{equation} then our condition $\sup_n
V_n^{(\varrho,0)} Q_n^{(\varrho,0)} <\infty$ and \eqref{Vrhoest}
imply that there exists a constant $C$ such that
\begin{equation}\label{keyest} \sum_{l=1}^{j} V_{n_l} \varrho_{n_l}^{-1}
\le C V_{n_{j+1}} \varrho_{n_{j+1}}^{-1}. \end{equation} If, on the
other hand, we have
\[ Q_{n_j}-Q_{n_{j+1}} < \frac{\varepsilon}{V_{n_{j+1}}}, \]
then an application of \eqref{localize} of Lemma~\ref{varrhoest}
gives $ \varrho_{n_{j+1}}/\varrho_{n_j} \le 5/4$ if $\varepsilon$ is
sufficiently small. Hence we have
\[ \frac{V_{n_{j+1}}\varrho_{n_{j}}}{V_{n_j}\varrho_{n_{j+1}}} \ge
\frac{8}{5},
\]
which means that $V_{n_j}\varrho_{n_j}^{-1}$ increases exponentially
on any set of consecutive integers $j$ for which \eqref{onone}
fails. Combining \eqref{keyest} with the latter estimate, we
therefore get that \[ \sum_{l=1}^{j} V_{n_l} \varrho_{n_l}^{-1} \le
\left(\frac{5}{8}C+\frac{8}{3}\right) V_{n_{j+1}}
\varrho_{n_{j+1}}^{-1}\] when \eqref{onone} fails and $\varepsilon$
is sufficiently small. Thus \eqref{keyest} holds for every index $j$
if the constant $C$ is suitably adjusted. Hence, by part (i) of
Lemma~\ref{simplefacts}, there exists a constant $C$ such that
\[ \frac{\varrho_{n_{j+l}}}{\varrho_{n_j}} \le C
\frac{V_{n_{j+l}}}{V_{n_j}} 2^{-\delta l} \le C
\left(\frac{V_{n_{j+l}}}{V_{n_j}}\right)^{1-\delta/2},\] where in
the last step we used that $V_{n_{j+1}}/V_{n_j}\le 4$ for
sufficiently large $j$. We are done since it suffices to establish
\eqref{boundrho} for $n=n_j$ and $m=n_{j+l}$.

(1) \emph{$\Lambda$ is a $v$-perturbation of $\Gamma$ of deficiency
$1$.} As in the previous case, we begin by finding estimates for the
weight sequences $\nu=(\nu_n)$ and $\varpi=(\varpi_j)$ appearing in
Theorem~\ref{th13}. If $\Lambda$ is a $v$-perturbation of $\Gamma$
of deficiency $1$ and an exact uniqueness sequence for
$H_{(\Gamma,v)}$, then there is a constant $c$ such that
\[ \sum_{n=1}^\infty \frac{e_n v_n}{z-\gamma_n}
=\frac{c}{(z-\gamma_1)(z-\gamma_2)} \prod_{m=3}^\infty
\frac{1-z/\lambda_m}{1-z/\gamma_n} \] for every $z$ in $\CC\setminus
\Gamma$, where again $e=(e_n)$ is the vector such that
$H_{(\Gamma,v);(\Lambda,w)}e=(1,0,0,...)$. Arguing in the same way
as in the preceding case, we obtain from this relation the estimates
\begin{equation}\label{nuest2}\nu_n\simeq w_n \varrho_n |\gamma_n|^{-2}
\end{equation}
and
\begin{equation}\label{varpiest2} \varpi_n\simeq v_n \varrho_n^{-1}|\gamma_n|^{2}.
\end{equation}

We now set
\[ V^{(\varrho,1)}_n=\sum_{m=1}^{n-1} v_n |\gamma_n|^{2} \varrho^{-1}_n \ \ \
\text{and} \ \ \ P^{(\varrho,1)}_n=\sum_{m=n+1}^{\infty}
v_n\varrho_n \] as well as
\[ W^{(\varrho,1)}_n=\sum_{m=1}^{n-1} w_n |\gamma_n|^{-2} \varrho_n \ \ \
\text{and} \ \ \ Q^{(\varrho,1)}_n=\sum_{m=n+1}^{\infty}
w_n|\gamma_n|^{-4} \varrho_n.
\]
By Theorem~\ref{th13} and Theorem~\ref{th1}, we must have $\sup_n
W_n^{(\varrho,1)} P_n^{(\varrho,1)} <\infty$; we will now show that
also the estimate in part (1) is a consequence of this condition.

We let the sequence $(n_j)_j$ be as above and find that
\begin{equation}\label{dualest} P_{n_j}^{(\varrho,1)} \simeq
 \sum_{l=j+1}^{\infty} V_{n_l} \varrho_{n_l}^{-1}
\end{equation}
whenever $j\ge 1$. Now if
\begin{equation} \label{onone2} Q_{n_{j+1}}-Q_{n_{j}} \ge
\frac{\varepsilon}{V_{n_{j+1}}},
\end{equation} then it follows from the condition $\sup_n
W_n^{(\varrho,1)} P_n^{(\varrho,1)} <\infty$ and \eqref{dualest}
that
\begin{equation}\label{keyest2}
\sum_{l=j+1}^{\infty} V_{n_l} \varrho_{n_l}^{-1} \lesssim V_{n_j}
\varrho_{n_j}^{-1}. \end{equation} As in the preceding case, we find
that, if $\varepsilon$ is sufficiently small, then
$V_{n_j}\varrho_{n_j}^{-1}$ increases exponentially on any set of
consecutive integers $j$ for which \eqref{onone2} fails. The
relation \eqref{dualest} implies that no such set is infinite; thus
there is an infinite sequence of indices $n_j$ for which
\eqref{keyest2} holds, and there must in fact be a uniform bound on
the number of points found in any set of consecutive integers $j$
for which \eqref{onone2} fails. We may infer from this argument that
in fact \eqref{keyest2} holds for every index $n_j\ge 1$. Finally,
we invoke part (ii) of Lemma~\ref{simplefacts}, which implies that
there is a constant $C$ such that
\[ \frac{\varrho_{n_{j+l}}}{\varrho_{n_j}} \ge C
\frac{V_{n_{j+l}}}{V_{n_j}} 2^{\delta l} \ge C
\left(\frac{V_{n_{j+l}}}{V_{n_j}}\right)^{1+\delta},\] and we are
done since it suffices to establish \eqref{boundrho2} for $n=n_j$
and $m=n_{j+l}$.

\subsection{Proof of the sufficiency of the conditions in
Theorem~\ref{y_int_th}}

We begin by noting that the condition $\sup_n V_n Q_n<\infty$
implies that $H_{(\Gamma,v);(\Lambda,w)}$ is a bounded
transformation. Indeed, \eqref{trivial} in Theorem~\ref{th1} holds
trivially when $\mu=\sum_n w_n\delta_{\lambda_n}$. We also have
\[ W_n \lesssim \frac{|\gamma_n|^2}{V_n} \ \ \ \text{and} \ \ \
P_n \lesssim \frac{v_n}{|\gamma_n|^2} \] by the assumptions that
$v_n=o(V_n)$ and $\sup_n V_n Q_n <\infty$. Therefore,
Theorem~\ref{th1} allows us to conclude that
$H_{(\Gamma,v);(\Lambda,w)}$ is a bounded transformation.

We will now use Theorem~\ref{th13} and show that the respective
conditions in part (0) and part (1) in Theorem~\ref{y_int_th} imply
those in Theorem~\ref{th13}. The sequence $(n_j)_j$ will be the same
as in the previous subsection.

(0) \emph{$\Lambda$ is an exact $v$-perturbation of $\Gamma$.} We
already know from Lemma~\ref{unique1} that if $\Lambda$ is an exact
$v$-perturbation of $\Gamma$, then $\Lambda$ is a uniqueness
sequence for $H_{(\Gamma,v)}$. To check that $\Lambda$ is in fact an
exact uniqueness sequence for $H_{(\Gamma,v)}$, we note that we may
write
\[
\frac{c}{z-\gamma_1} \prod_{m=2}^\infty
\frac{1-z/\lambda_m}{1-z/\gamma_n}=\sum_{n=1}^\infty \frac{a_n
v_n}{z-\gamma_n}+ h(z),
\]
where $h$ is an entire function and \[ |a_n|^2 v_n \simeq
\frac{w_n}{|\gamma_n|^2} \varrho_n. \] By the assumption that
$\sup_n V_n Q_n <\infty$, we have
\[ \sum_{n=1}^\infty |a_n|^2 v_n \lesssim \sum_{j=1}^\infty
\frac{\varrho_{n_j}}{V_{n_j}}, \] which, in view of
\eqref{boundrho}, implies that $(a_n)$ is in $\ell^2_v$. In
particular, we then have
\[ |h(z)|^2 \lesssim \frac{\varrho_n}{|z|^2} + \frac{V_n}{|z|^2}
\]
when $z$ is in $D_n(v; M)$ with $M$ sufficiently large. Thus
$h(z)\to 0$ when $z\to \infty$ which means that $h\equiv 0$.

It remains only to verify that $H_{(\Lambda,\varpi);(\Gamma,\nu)}$
is a bounded transformation. By Theorem~\ref{th1}, we need to show
that we have both $\sup_n W_n^{(\varrho,0)}P_n^{(\varrho,0)} <\infty
$ and $\sup_n V_n^{(\varrho,0)}Q_n^{(\varrho,0)}<\infty$. To this
end, we note that since $\varrho_n$ can only grow sub-exponentially,
we have $\sup_n W_n^{(\varrho,0)}P_n^{(\varrho,0)} <\infty $ by the
same argument that gave $\sup_n W_n P_n <\infty$. Since $\sup_n V_n
Q_n <\infty$, we have
\[ V_n^{(\varrho,0)}Q_n^{(\varrho,0)} \lesssim \sum_{n_j<n}
\frac{V_{n_j}}{\varrho_{n_j}} \frac{\varrho_n}{V_n}; \] here the
right-hand side is uniformly bounded whenever \eqref{boundrho}
holds.

(1) \emph{$\Lambda$ is a $v$-perturbation of $\Gamma$ of deficiency
$1$.} In view of Lemma~\ref{unique3}, we will have that $\Lambda$ is
an exact uniqueness sequence for $H_{(\Gamma,v)}$ if we can show
that there is no nonzero $a$ in $\ell^2_v$ such that
$H_{(\Gamma,v)}a$ vanishes on $\Lambda$. To show this, we assume to
the contrary that such a sequence $a$ exists. Then there is a
constant $c$ such that
\begin{equation} \label{formula}\sum_{n=1}^\infty \frac{a_n v_n}
{z-\gamma_n}= \frac{c}{z-\gamma_1}
 \prod_{m=2}^\infty
\frac{1-z/\lambda_m}{1-z/\gamma_n}.  \end{equation} By estimating
each side of \eqref{formula} for $z$ in $D_n(v; M)$ with $M$
sufficiently large, we get \[ V_n \sum_{m=1}^\infty |a_m|^2 \gtrsim
\varrho_n.
\] But this is a contradiction, because \eqref{boundrho2} implies that
$\varrho_n/V_n$ is an increasing sequence.

It remains only to verify that $H_{(\Lambda,\varpi);(\Gamma,\nu)}$
is a bounded transformation. To this end, we note that $\sup_n
V_n^{(\varrho,1)}Q_n^{(\varrho,1)} <\infty $ holds trivially because
$1/\varrho_n$ can only grow sub-exponentially, while
\[ W_n^{(\varrho,1)}P_n^{(\varrho,1)} \lesssim \sum_{n_j<n}
\frac{P_{n_j}}{\varrho_{n_j}} \frac{\varrho_n}{P_n}, \] which is
uniformly bounded when \eqref{boundrho2} holds.

\end{document}